\documentclass{amsart}

\usepackage{template}
\usepackage{quiver}

\newcommand\blfootnote[1]{%
  \begingroup
  \renewcommand\thefootnote{}\footnote{#1}%
  \addtocounter{footnote}{-1}%
  \endgroup
}

\title[Reedy Fibrant Replacements of Projective Fibrant Simplicial Presheaves]{Constructing Reedy Fibrant Replacements of Projective Fibrant Simplicial Presheaves}
\author{Jack Rom{\"o} \\
\textit{Email:} \texttt{\href{mailto:jack.t.romo@gmail.com}{jack.t.romo@gmail.com}}}

\begin{document}

\begin{abstract}
    In this paper, we construct an explicit Reedy fibrant replacement functor for projective fibrant simplicial presheaves $X : \mathscr{C}^{op} \rightarrow \textbf{sSet}$, where $\mathscr{C}$ is a Reedy category. Our approach describes, by hand, all latching maps for the Reedy fibrant replacement by an inductive series of higher homotopies. We explore the nature of our functor by using it to recover some standard homotopy limit constructions.
\end{abstract}
\maketitle

\section{Introduction}

\blfootnote{\hspace{8pt}Supported by the Leeds Doctoral Scholarship from the Faculty of Engineering and Physical Sciences at the University of Leeds, 2020-2024.}
Fibrant replacement is a fundamental operation in model category theory and more general homotopy theory. In this article, we take a particular interest in the case of \emph{Reedy fibrant replacement}, that is fibrant replacement in a Reedy model category $\mathscr{M}^\mathscr{C}$ for a model category $\mathscr{M}$ and Reedy category $\mathscr{C}$. This operation is used for instance in the construction of homotopy limits, though has a greater scope of application when working in a Reedy model category as the fibrant-cofibrant objects are in general better behaved and thus desirable to work with. Many Reedy categories, such as the simplicial category $\Delta$ and the category of $n$-dimensional pasting diagrams $\Theta_n$ are moreover \emph{elegant Reedy}, as defined in \cite{bergnerRezkReedyCategoriesVarTheta2013}, in which case all functors are Reedy cofibrant. Thus, there are many scenarios when one may ignore cofibrant replacement, such as for model categories of Reedy fibrant (complete) $n$-fold Segal spaces and $\Theta_n$-spaces.

The author of this article takes a special interest in the case of \emph{$n$-fold Segal spaces}, that is functors
$$
    X : (\Delta^{op})^n \rightarrow \textbf{sSet}
$$
that are levelwise fibrant and satisfy a series of conditions allowing them to be used as a model for $(\infty, n)$-categories. For instance, given any $i_1, \cdots, i_n \geq 0$, the maps
\begin{align}
    X_{i_1, \cdots, i_j, \cdots, i_n} \rightarrow X_{i_1, \cdots, 1, \cdots, i_n} \times^h_{X_{i_1, \cdots, 0, \cdots, i_n}} \cdots \times^h_{X_{i_1, \cdots, 0, \cdots, i_n}} X_{i_1, \cdots, 1, \cdots, i_n}
\end{align}
must be weak equivalences of Kan complexes. $n$-fold Segal spaces can either be \emph{projective fibrant}, meaning levelwise fibrant, or \emph{Reedy fibrant}. The former case has been used to construct models for $(\infty, n)$-categories of manifolds and cobordisms relevant to topological quantum field theories \cite{calaqueScheimbauerNoteInftyCategory2019, lurieClassificationTopologicalField2009}. The latter, due to the maps in (1) being fibrations, allows in the case $n = 2$ for the direct construction of a \emph{homotopy bicategory} $h_2(X)$ for any Reedy fibrant $2$-fold Segal space $X$ \cite{romoHomotopyBicategories2Fold2023}. Such a construction amounts to collapsing all higher morphisms in an $(\infty, 2)$-category down to path components, resulting in a bicategory. This construction should help compare developments in TQFTs in the worlds of symmetric monoidal $2$-fold Segal spaces, such as in \cite{lurieClassificationTopologicalField2009}, \cite{calaqueScheimbauerNoteInftyCategory2019} and \cite{gradyPavlovGeometricCobordism2022}, and symmetric monoidal bicategories, such as \cite{schommer-priesClassificationTwoDimensionalExtended2014}, \cite{pstragowskiDualizableObjectsMonoidal2014} and \cite{bartlettDouglasSchommerPriesVicaryExtended3DimensionalBordismTheory2014}. To do so requires taking explicit Reedy fibrant replacements of projective fibrant $2$-fold Segal spaces of manifolds and cobordisms present in the literature.

There are many ways to obtain Reedy fibrant replacements present in the literature. The definition of a Reedy model structure has this information baked into a straightforward weak factorization system \cite[pg. 293]{hirschhornModelCategoriesTheir2009}; it can be argued that this is one of the core reasons why Reedy categories exist at all. One can also make use of a small object argument \cite[ch. 12]{riehlCategoricalHomotopyTheory2014}. Some specific cases of Reedy fibrant replacement also exist in the literature for particular applications \cite{ducoulombierFresseTurchinProjectiveReedy2022} \cite[Lemma 6.4.1]{stenzelUnivalenceRezkCompleteness2019}. Unfortunately, none of these approaches are designed to address our requirements. The latter two handle specific cases,  while applying the small-object argument presents a great challenge in understanding the resulting fibrant replacement levelwise, which is not ideal if one seeks to explicitly construct a homotopy bicategory from a given Reedy fibrant replacement. If we wish to use the general weak factorization system for Reedy model structures, we will have to make choices of factorizations in the underlying model category in order to obtain an explicit construction. To this end, completely general methods to factorize maps of simplicial sets, such as employing the small object argument, oftentimes present a challenge in understanding the resulting Reedy fibrant replacements of explicit examples levelwise.

In this paper, we present a functorial construction of a Reedy fibrant replacement $R(X)$ for any projective fibrant simplicial presheaf $X : \mathscr{C} \rightarrow \textbf{sSet}$, where $\textbf{sSet}$ has the Quillen model structure and $\mathscr{C}$ is Reedy.

More precisely, let $(\textbf{sSet}^\mathscr{C})_{proj}$ be the category of projective fibrant functors $\mathscr{C} \rightarrow \textbf{sSet}$. We will prove the following theorem:

\begin{theorem}
    The construction produced inductively in Definitions \ref{def:reedy:inductive0}, \ref{def:reedy:r_constr} and \ref{def:reedy:kappa_inductive} establishes a diagram of functors and natural transformations of the form
    \[\begin{tikzcd}
    	{(\textbf{sSet}^\mathscr{C})_{proj}} && {(\textbf{sSet}^\mathscr{C})_{proj}}
    	\arrow[""{name=0, anchor=center, inner sep=0}, "R"', curve={height=18pt}, from=1-1, to=1-3]
    	\arrow[""{name=1, anchor=center, inner sep=0}, "id", curve={height=-18pt}, from=1-1, to=1-3]
    	\arrow["\kappa"', shorten <=5pt, shorten >=5pt, Rightarrow, from=1, to=0]
    \end{tikzcd}\]
    such that:
    \begin{enumerate}
        \item for any $X \in (\textbf{sSet}^\mathscr{C})_{proj}$, $R(X)$ is Reedy fibrant;

        \item $\kappa_X : X \rightarrow R(X)$ is a levelwise trivial cofibration;

        \item $R$ preserves levelwise weak equivalences and finite products.
    \end{enumerate}
\end{theorem}

In the case that $\mathscr{C}$ is \emph{elegant Reedy}, $\kappa_X$ will moreover be a trivial Reedy cofibration.

This construction is applied in the author's thesis \cite{romoTowardsAlgebraicNCategories} to obtain the homotopy bicategory of a projective fibrant $2$-fold Segal space. We hope that our work may find similar applications for other models or cases of higher category; we expect our approach will be similarly effective for $\Theta_n$-spaces \cite{rezkCartesianPresentationWeak2010} and the various models of $(\infty, n)$-category encompassed by Barwick and Schommer-Pries' axioms \cite{barwickSchommerPriesUnicityHomotopyTheory2020}.

The paper begins with some background on constructing homotopy pullbacks via Brown's factorization lemma. Next, we discuss some techniques needed to build the higher homotopies necessary for the latching maps in our fibrant replacement functor. We then turn to the construction itself, which is inductive in degree by nature. Finally, we note some useful properties of our construction.

\textbf{Acknowledgements.} The author would like to thank his PhD advisors at the University of Leeds, Jo\~ao Faria Martins and Nicola Gambino, for their support and assistance throughout this project. The author would also like to thank Mike Shulman and Raffael Stenzel for useful e-mail correspondences, Julie Bergner for useful discussions and Charles Rezk and Lukas Waas for helpful discussions during the Panorama of Homotopy Theory conference at the University of Oxford in June 2023.

The contents of this paper are largely taken vertabim from the author's Ph.D. thesis \cite{romoTowardsAlgebraicNCategories}, with minor modifications.
\section{Notation and Conventions}

We will write $\textbf{sSet}$ for the category of simplicial sets and $\Delta$ for the simplicial category. We write $\textbf{sSpace}_n := \textbf{sSet}^{(\Delta^{op})^n}$ for the category of \emph{$n$-uple simplicial spaces}. Let $\textbf{nerve} : \textbf{Cat} \rightarrow \textbf{sSet}$ be the quasicategorical nerve functor.
\section{Homotopy Pullbacks and Path Spaces}

We will have need for a suitable notion of \emph{homotopy pullback} for cospans of simplicial sets. We begin with a consideration of general path spaces in a model category:

\begin{definition}[{\cite[pg. 22]{dwyerSpalinskiHomotopyTheoriesModel1995}}]
    Let $\mathscr{M}$ be a model category and $C \in \mathscr{M}$. A \emph{path object} for $C$ is a factorization
    $$
        C \xrightarrow{d} \mathcal{P}(C) \xrightarrow{p} C \times C
    $$
    of the diagonal map $C \rightarrow C \times C$ such that $d$ is a weak equivalence. It is a \emph{good path object} if $p$ is a fibration and a \emph{very good path object} if, moreover, $d$ is a trivial cofibration.
\end{definition}

The purpose of path objects is often to obtain \emph{right homotopies}, a notion of homotopy amenable to model categories:

\begin{definition}[{\cite[pg. 22]{dwyerSpalinskiHomotopyTheoriesModel1995}}]
    A \emph{right homotopy} $H$ from $f$ to $g$ for $f, g : C \rightarrow D$ in a model category $\mathscr{M}$ is a map
    $$
        H : C \rightarrow \mathcal{P}(D)
    $$
    for some path object $\mathcal{P}(D)$ of $D$ such that the two compositions $C \rightarrow \mathcal{P}(D) \rightarrow D$ result in $f$ and $g$ respectively. It is called a \emph{good} or \emph{very good} right homotopy if $\mathcal{P}(D)$ is a good or very good path object, respectively.
\end{definition}

We will encounter so-called \emph{left homotopies} later in this thesis, which are somehow dual to right homotopies.

There are two path objects we will take an interest in for $\textbf{sSet}$, neither of which we claim originality for. The first of these appears for instance in \cite[Def. 1.1.23]{riehlVerityElementsCategoryTheory2022}. Before we present it, we need a quick definition:

\begin{definition} \label{defn_h1_i_n}
    Let $I[n]$ be the groupoid whose object set is $[n]$ and with all hom-sets singletons.
\end{definition}

For some topological connection to this object, consider the following standard construction:

\begin{definition}[{\cite[ch. 2]{may1999AConciseCourse}}] \label{defn_hpb_pi1}
    The \emph{fundamental groupoid} functor $\Pi_1 : \textbf{Top} \rightarrow \textbf{Grpd}$ sends a topological space $X$ to the groupoid $\Pi_1(X)$ whose objects are points in $X$ and whose morphisms are homotopy classes of paths relative to start and end points.
\end{definition}

Then, in particular, one might notice that $I[n]$ is isomorphic to a full subcategory of the fundamental groupoid $\Pi_1(\Delta_t[n])$ of the $n$-simplex whose objects are the `corners' of $\Delta_t[n]$, namely those points $(x_0, \cdots, x_n)$ where some $x_i = 1$. This is equivalent to the entirety of $\Pi_1(\Delta_t[n])$ along the inclusion functor, a fact quickly proven by the $n$-simplex being contractible. We might therefore consider $I[1]$ as a finite model for a contractible homotopy type with two points and $I[n]$ similarly for $n+1$ points.

The extension of these ideas to simplicial sets for us will be the space $\textbf{nerve}(I[1])$:

\begin{definition}
    Suppose $X \in \textbf{sSet}$ is a Kan complex. Then define $I := \textbf{nerve}(I[1])$ and in turn $X^I$ to be the space
    $$
        X \xrightarrow{c_X} X^I := X^{\textbf{nerve}(I[1])} \xrightarrow{\langle s_X, t_X \rangle} X \times X
    $$
    where the maps $c_X : X \rightarrow X^I$ and $s_X, t_X : X^I \rightarrow X$ are given by applying $X^{\textbf{nerve}(-)}$ to the natural maps
    $$
        \ast \sqcup \ast \rightarrow I[1] \rightarrow \ast
    $$
    in $\textbf{Cat}$. We will refer to $c_X$ as the \emph{constant path map}, $s_X$ as the \emph{source map} and $t_X$ as the \emph{target map}.
\end{definition}

Our second path object is built upon the pushout of the natural span $\textbf{nerve}(I[1]) \leftarrow \ast \rightarrow \textbf{nerve}(I[1])$, where the leftmost map identifies the object $1$ while the rightmost map identifies $0$:

\begin{definition}
    Suppose $X \in \textbf{sSet}$ is a Kan complex. Then define
    $\Lambda := \textbf{nerve}(I[1]) \sqcup_\ast \textbf{nerve}(I[1])$ and in turn $X^\Lambda$ to be the space
    $$
        X \xrightarrow{\lambda_X} X^\Lambda := X^{\textbf{nerve}(I[1]) \sqcup_\ast \textbf{nerve}(I[1])} \xrightarrow{\langle \sigma_X, \tau_X \rangle} X \times X
    $$
    where the maps $\lambda_X : X \rightarrow X^\Lambda$ and $\sigma_X, \tau_X : X^\Lambda \rightarrow X$ are given by the natural maps
    $$
        \ast \sqcup \ast \rightarrow \textbf{nerve}(I[1]) \sqcup_\ast \textbf{nerve}(I[1]) \rightarrow \ast
    $$
    in $\textbf{sSet}$. We will refer to $\lambda_X$ as the \emph{constant path map}, $\sigma_X$ as the \emph{source map} and $\tau_X$ as the \emph{target map}.
\end{definition}

Note that $X^\Lambda$ does not result from some diagram in $\textbf{Cat}$; indeed, appending two copies of $I[1]$ along $\ast$ in $\textbf{Cat}$ would yield $I[2]$ instead, rather than the above simplicial set. We can see $\Lambda$ as a `horn' of sorts, though where each 1-simplex has an inverse.

\begin{proposition} \label{prop:reedy:xi_path_obj}
    If $X$ is a Kan complex, then $X^I$ is a very good path object for $X$.
\end{proposition}

\begin{proof}
    The map $X^I \rightarrow X \times X$ being a fibration follows from the inclusion of objects $\ast \sqcup \ast \hookrightarrow \textbf{nerve}(I[1])$ being a cofibration and \cite[Ex. 9.1.13]{hirschhornModelCategoriesTheir2009}, which states that $\textbf{sSet}$ has a natural simplicial model structure.

    The map $X \rightarrow X^I$ is clearly a cofibration, as it is a levelwise inclusion. To show it is a weak equivalence, note that each map $\ast \rightarrow \textbf{nerve}(I[1])$ is a trivial cofibration, so that the two maps $X^I \rightarrow X$ are trivial fibrations. The composites $X \rightarrow X^I \rightarrow X$ are the identity, so by 2-out-of-3 we have a trivial cofibration as needed.
\end{proof}

\begin{proposition}
    If $X$ is a Kan complex, then $X^\Lambda$ is a very good path object for $X$.
\end{proposition}

\begin{proof}
    The proof is similar to $X^I$. Again, $X^\Lambda \rightarrow X \times X$ is a fibration since $\ast \sqcup \ast \hookrightarrow \Lambda$ is a cofibration. Moreover, the constant path map $X \rightarrow X^\Lambda$ is immediately a cofibration. For it to be trivial, it again suffices to show the two maps $\ast \rightarrow \Lambda$ are trivial cofibrations. Considering the pushout diagram
    \[\begin{tikzcd}
    	& \ast & {\textbf{nerve}(I[1])} & \ast \\
    	\ast & {\textbf{nerve}(I[1])} & \Lambda
    	\arrow[from=1-2, to=1-3]
    	\arrow["i"', from=1-2, to=2-2]
    	\arrow["j", from=1-3, to=2-3]
    	\arrow["k"', from=1-4, to=1-3]
    	\arrow[from=2-1, to=2-2]
    	\arrow[from=2-2, to=2-3]
    \end{tikzcd}\]
    we see that since $i$ is a trivial cofibration, $j$ must be one too. Since $k$ is a trivial cofibration, the composite $jk$ is a trivial cofibration as needed. A similar story holds for the horizontal maps.
\end{proof}

The method to construct homotopy pullbacks we employ is standard; indeed, given any good path object in a model category, one can produce a model of homotopy pullback from said path object in a similar way \cite[Cor. 3.4]{nlabHomotopyPullback}. We will however take a perhaps unusual definition of `homotopy pullback' in $\textbf{sSet}$ within the confines of this approach, electing to use $\Lambda$ rather than $I$. We do not require much specific insight into homotopy limits; for those interested, one may consult \cite[ch. 18-19]{hirschhornModelCategoriesTheir2009} or \cite[ch. 5]{riehlCategoricalHomotopyTheory2014}.

\begin{definition}
    Suppose $f : X \rightarrow Y \leftarrow Z : g$ is a cospan of Kan complexes in $\textbf{sSet}$. Define the \emph{homotopy pullback} of this cospan to be
    $$
        X \times^h_Y Z := X \times_Y Y^\Lambda \times_Y Z,
    $$
    where the limit is taken over the diagram
    $$
        X \xrightarrow{f} Y \xleftarrow{\sigma_Y} Y^\Lambda \xrightarrow{\tau_Y} Y \xleftarrow{g} Z.
    $$
\end{definition}

A more standard definition would perhaps be $X \times_Y Y^I \times_Y Z$ or something analogous, such as $X \times_Y Y^{\textbf{Sing}([0, 1])} \times_Y Z$, which are both perfectly serviceable in the scenario of Kan complexes. All of these differ solely by the choice of path object used. We employ the above as it will be the easiest to cohere with the results of our upcoming computations.

As a show of good faith, we should prove this is indeed a reasonable notion of homotopy pullback. We will first need a result about factorizing maps between Kan complexes that will prove crucial in our work on Reedy fibrant replacement later. This result is entirely standard; for instance, both its statement and much of its proof are a special case of the \emph{factorization lemma} in \cite{brownAbstractHomotopyTheoryGeneralized1973}. Brown states in the proof thereof that this derives from standard methods in homotopy theory.

\begin{lemma}\label{lemma:reedy:sset_path_fib}
    Suppose $f : X \rightarrow Y$ is a map in $\textbf{sSet}$ between Kan complexes $X$ and $Y$. Then the induced map
    $$
        X \times_Y Y^{\textbf{nerve}(I[1])} \rightarrow Y
    $$
    by the uppermost horizontal maps in the diagram
    \[\begin{tikzcd}
    	{X \times_Y Y^{\textbf{nerve}(I[1])}} & {Y^{\textbf{nerve}(I[1])}} & Y \\
    	X & Y
    	\arrow[from=1-1, to=1-2]
    	\arrow[from=1-1, to=2-1]
    	\arrow["\lrcorner"{anchor=center, pos=0.125}, draw=none, from=1-1, to=2-2]
    	\arrow["{t_Y}", from=1-2, to=1-3]
    	\arrow["{s_Y}", from=1-2, to=2-2]
    	\arrow["f"', from=2-1, to=2-2]
    \end{tikzcd}\]
    is a fibration. Moreover, the map $X \rightarrow X \times_Y Y^{\textbf{nerve}(I[1])}$, given by $1_X$ and $X \xrightarrow{f} Y \xrightarrow{c_Y} Y^{\textbf{nerve}(I[1])}$, is a trivial cofibration.
\end{lemma}

\begin{proof}
    Consider the diagram
    \[\begin{tikzcd}
    	{X \times_Y Y^{\textbf{nerve}(I[1])}} & {(X \times Y) \times_{Y \times Y} Y^{\textbf{nerve}(I[1])}} & {Y^{\textbf{nerve}(I[1])}} \\
    	Y & {X \times Y} & {Y \times Y}
    	\arrow["\cong", from=1-1, to=1-2]
    	\arrow["t"', two heads, from=1-1, to=2-1]
    	\arrow[from=1-2, to=1-3]
    	\arrow[two heads, from=1-2, to=2-2]
    	\arrow["\ulcorner"{anchor=center, pos=0.125}, draw=none, from=1-2, to=2-3]
    	\arrow[two heads, from=1-3, to=2-3]
    	\arrow[two heads, from=2-2, to=2-1]
    	\arrow["{f \times 1_Y}"', from=2-2, to=2-3]
    \end{tikzcd}\]
    By Proposition \ref{prop:reedy:xi_path_obj}, the rightmost vertical map is a fibration, so the pullback map is a fibration. Projections from products of fibrant objects are also fibrations, meaning the map $t$ is a fibration as needed.

    Note then that the composite map $X \rightarrow X \times_Y Y^{\textbf{nerve}(I[1])} \rightarrow X$ is $1_X$. Because the map $X \times_Y Y^{\textbf{nerve}(I[1])} \rightarrow X$ is a trivial fibration, by 2-out-of-3 the map $X \rightarrow X \times_Y Y^{\textbf{nerve}(I[1])}$ is a weak equivalence. That it is a cofibration is immediate.
\end{proof}

\begin{proposition}
    Let $f : X \rightarrow Y \leftarrow Z : g$ be a cospan of Kan complexes in $\textbf{sSet}$. Then $X \times^h_Y Z$ is a homotopy pullback of this cospan.
\end{proposition}

\begin{proof}
    It suffices to prove that the cospan
    $$
        X \times_Y Y^I \rightarrow Y \leftarrow Y^I \times_Y Z
    $$
    is injective fibrant as a diagram $\mathscr{K} \rightarrow \textbf{sSet}$, where $\mathscr{K}$ is the natural cospan diagram.
    
    Note that $\mathscr{K}$ may be interpreted as a Reedy category $a \rightarrow b \leftarrow c$, where $\textbf{deg}(a) = \textbf{deg}(c) = 1$ and $\textbf{deg}(b) = 0$. As there are no maps of positive degree, the Reedy cofibrations in $\textbf{sSet}^\mathscr{K}$ are injective cofibrations, so Reedy and injective fibrancy coincide. It thus suffices to prove that the maps $X \times_Y Y^I \rightarrow Y$ and $Y^I \times_Y Z \rightarrow Y$ are fibrations, which is immediate by Lemma \ref{lemma:reedy:sset_path_fib}.
\end{proof}
\section{Paths in Path Spaces}

For the purposes of this chapter, we will find utility in a few minor constructions worth establishing before we proceed further. In particular, we will need a standard homotopy of path spaces that retracts a path down to its source. We visualize this homotopy as, given a path $p: \textbf{nerve}(I[1]) \rightarrow X$, producing a path of paths $D_X(p) : \textbf{nerve}(I[1]) \rightarrow X^{\textbf{nerve}(I[1])}$ of the form
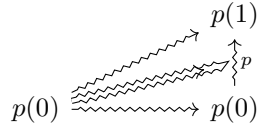
\[\begin{tikzcd}
	&& {p(1)} \\
	{p(0)} && {p(0)}
	\arrow[squiggly, from=2-1, to=1-3]
	\arrow[squiggly, from=2-1, to=2-3]
	\arrow[""{name=0, anchor=center, inner sep=0}, "p"', squiggly, from=2-3, to=1-3]
	\arrow[shorten >=11pt, Rightarrow, squiggly, from=2-1, to=0]
\end{tikzcd}\]
More formally, this can be given by a map
$$
     \textbf{nerve}(I[1]) \times \textbf{nerve}(I[1]) \cong \textbf{nerve}(I[1] \times I[1]) \xrightarrow{\textbf{nerve}(Q)} \textbf{nerve}(I[1]) \xrightarrow{p} X
$$
where $Q : I[1] \times I[1] \rightarrow I[1]$ is the map sending $(i, j) \mapsto i \times j$ for $i, j \in \{0, 1\} = \textbf{ob}(I[1])$. Visually, if we depict $I[1] \times I[1]$ as the category
\[\begin{tikzcd}
	{(0, 0)} && {(1, 0)} \\
	{(0, 1)} && {(1, 1)}
	\arrow["\cong", from=1-1, to=1-3]
	\arrow["\cong"', from=1-1, to=2-1]
	\arrow["\cong", from=1-3, to=2-3]
	\arrow["\cong"', from=2-1, to=2-3]
\end{tikzcd}\]
then its image under $Q$ corresponds to the diagram in $I[1]$ of the form
\[\begin{tikzcd}
	0 && 0 \\
	0 && 1
	\arrow["{1_0}", from=1-1, to=1-3]
	\arrow["{1_0}"', from=1-1, to=2-1]
	\arrow["\cong", from=1-3, to=2-3]
	\arrow["\cong"', from=2-1, to=2-3]
\end{tikzcd}\]
This evidently induces the desired homotopy of paths; the image of the leftmost vertical map under postcomposition with $p$ will be the constant path on $p(0)$, while the image of the rightmost will be $p$ itself.

We formalize this discussion as follows:

\begin{definition}
    Suppose $B \in \textbf{sSet}$. Then define
    $$
        D_B : B^{\textbf{nerve}(I[1])} \rightarrow \big( B^{\textbf{nerve}(I[1])} \big)^{\textbf{nerve}(I[1])}
    $$
    to be the map defined by precomposing the isomorphism
    $$
        B^{\textbf{nerve}(I[1] \times I[1])} \cong \big( B^{\textbf{nerve}(I[1])} \big)^{\textbf{nerve}(I[1])}
    $$
    with $B^{\textbf{nerve}(Q)}$, where $Q : I[1] \times I[1] \rightarrow I[1]$ is the map defined such that 
    $$
        Q(i, j) := i \times j
    $$
    for $i, j \in \{0, 1\} = \textbf{ob}(I[1])$.
\end{definition}

\begin{proposition}
    $D^B$ defines a very good right homotopy to $1_{B^{\textbf{nerve}(I[1])}}$ from the map
    $$
        B^{\textbf{nerve}(I[1])} \xrightarrow{s_B} B \xrightarrow{c_B} B^{\textbf{nerve}(I[1])}.
    $$
\end{proposition}

\begin{proof}
    We consider the two projections
    $$
        B^{\textbf{nerve}(I[1])} \xrightarrow{D_B} \big( B^{\textbf{nerve}(I[1])} \big)^{\textbf{nerve}(I[1])} \rightarrow B^{\textbf{nerve}(I[1])}
    $$
    in turn, where the latter map is set either to $s_{B^{\textbf{nerve}(I[1])}}$ or $t_{B^{\textbf{nerve}(I[1])}}$. They are each given by precomposition with functors
    $$
        I[1] \rightarrow I[1] \times I[1] \xrightarrow{Q} I[1]
    $$
    where the first functor is either of the form $id \times \{0\}$ or $id \times \{1\}$. It is clear then that the first map sends $i \mapsto 0$, while the second sends $i \mapsto i$, as needed.
\end{proof}

We will make use of this homotopy repeatedly to prove that the constant maps $X \rightarrow R(X)$ given by constant paths are weak equivalences; in fact, we will construct a partial deformation retract thereof. Returning to the visual depiction of a $0$-simplex $(p, f, q)$ in $R(A)(1)$ for $A : \Delta^{op}_{\leq 1} \rightarrow \textbf{sSet}$ projective fibrant of the form
\[\begin{tikzcd}
	{x'} & x & y & {y'}
	\arrow["p", squiggly, from=1-1, to=1-2]
	\arrow["f", from=1-2, to=1-3]
	\arrow["q", squiggly, from=1-3, to=1-4]
\end{tikzcd}\]
we will produce a map $R(A)(1) \rightarrow A(1)$ sending $(p, f, q) \mapsto f$ and a homotopy built from $D_\bullet$ that deforms $R(A)(1) \rightarrow A(1) \rightarrow R(A)(1)$ to $1_{R(A)(1)}$, by gradually extending the paths $p$ and $q$ to their full lengths.

Another fact about $D_\bullet$ is also necessary:

\begin{proposition} \label{prop:reedy:db_precomp_const}
    The map
    $$
        B \xrightarrow{c_B} B^{\textbf{nerve}(I[1])} \xrightarrow{D_B} \Big( B^{\textbf{nerve}(I[1])} \Big)^{\textbf{nerve}(I[1])}
    $$
    is a constant homotopy from $c_B$ to $c_B$.
\end{proposition}

\begin{proof}
    Unwinding the definitions of $c_B$ and $D_B$ reveals this map is induced by the terminal map
    $$
        I[1] \times I[1] \rightarrow I[1] \rightarrow \ast.
    $$
\end{proof}
\section{Reedy Fibrant Replacement Functors}

Our goal in this section is to construct the Reedy fibrant replacement of a general levelwise fibrant functor
$$
    X : \mathscr{C} \rightarrow \textbf{sSet}
$$
for some Reedy category $\mathscr{C}$. That is to say, we seek a new functor $R(X) : \mathscr{C} \rightarrow \textbf{sSet}$ such that for all $c \in \mathscr{C}$, the map $R(X)(c) \rightarrow M_c R(X)$ is a fibration of simplicial sets, together with a levelwise weak equivalence $\kappa_X : X \rightarrow R(X)$.

In total, we will obtain a functor
$$
    R : (\textbf{sSet}^\mathscr{C})_{proj} \rightarrow (\textbf{sSet}^\mathscr{C})_{proj}
$$
where the domain and codomain are the full subcategories of \emph{projective}, namely \emph{levelwise}, fibrant functors, such that $R(X)$ is Reedy fibrant for all $X$. We will also find a natural transformation $\kappa : id \Rightarrow R$ such that for each $X$, $\kappa_X : X \to R(X)$ is levelwise a weak equivalence.

Moreover, setting $R^+$ to be the postcomposition with the functor $i^\ast$ defined by precomposition with $i : \mathscr{C}^+ \hookrightarrow \mathscr{C}$ of the form
\[\begin{tikzcd}
	{(\textbf{sSet}^{\mathscr{C}})_{proj}} && {(\textbf{sSet}^{\mathscr{C}})_{proj}} \\
	&& {(\textbf{sSet}^{\mathscr{C}^+})_{proj}}
	\arrow["R", from=1-1, to=1-3]
	\arrow["{R^+}"', from=1-1, to=2-3]
	\arrow["{i^\ast}", from=1-3, to=2-3]
\end{tikzcd}\]
and $\kappa^+ : i^\ast \Rightarrow R^+$ to be the whiskering
\[\begin{tikzcd}
	{(\textbf{sSet}^{\mathscr{C}})_{proj}} && {(\textbf{sSet}^{\mathscr{C}})_{proj}} & {(\textbf{sSet}^{\mathscr{C}^+})_{proj}}
	\arrow[""{name=0, anchor=center, inner sep=0}, "id", curve={height=-18pt}, from=1-1, to=1-3]
	\arrow[""{name=1, anchor=center, inner sep=0}, "R"', curve={height=18pt}, from=1-1, to=1-3]
	\arrow["{i^\ast}", from=1-3, to=1-4]
	\arrow["\kappa", shorten <=5pt, shorten >=5pt, Rightarrow, from=0, to=1]
\end{tikzcd}\]
we will then devise a natural transformation $\alpha : R^+ \Rightarrow i^\ast$ such that $\alpha \kappa^+ = 1_{i^\ast}$. For instance, if $\mathscr{C} = \Delta^{op}$, we have that the restriction of a simplicial space $X$ to $\mathscr{C}^+$ forgets the face maps and leaves us with only degeneracies, so that $R^+$, $\alpha$ and $\kappa^+$ would act on this restricted object.

While we will not necessarily have conversely that $\kappa^+ \alpha = 1_{R^+}$, we will construct a natural transformation
$$
    H : R^+ \Rightarrow (-)^{\textbf{nerve}(I[1])} \circ R^+
$$
which yields for each $X \in (\textbf{sSet}^{\mathscr{C}})_{proj}$ and each $c \in \mathscr{C}$ a right homotopy
$$
    (H_X)(c) : (\kappa^+_X)(c) \circ (\alpha_X)(c) \sim 1_{R^+(X)(c)}.
$$

The entirety of our construction largely hinges on Lemma \ref{lemma:reedy:sset_path_fib}. We will repeatedly apply this factorization in higher and higher degrees to force maps to matching objects to be fibrations. This alone does not explain how we will obtain maps from latching objects; to do so, we require $\alpha$ and $H$.

\begin{notation}
    Write $\mathscr{C}^{\leq n}$ for the full subcategory of a Reedy category $\mathscr{C}$ whose objects are those of degree less than or equal to $n$.
\end{notation}

In particular, note that $\mathscr{C}^{\leq 0}$ is a discrete category \cite[Ex. 15.1.23]{hirschhornModelCategoriesTheir2009}. We will henceforth write as a minor abuse of notation $i^\ast : (\textbf{sSet}^{\mathscr{C}^{\leq n}})_{proj} \rightarrow (\textbf{sSet}^{(\mathscr{C}^{\leq n})^+})_{proj}$ for all $n$.

\begin{notation}
    Suppose $X : \mathscr{C} \rightarrow \textbf{sSet}$ is functor from some category $\mathscr{C}$. Write $X^{\textbf{nerve}(I[1])} : \mathscr{C} \rightarrow \textbf{sSet}$ for the evident functor such that
    $$
        c \mapsto X(c)^{\textbf{nerve}(I[1])}
    $$
    with maps induced by postcomposition. 
\end{notation}

For the rest of this section, assume that $\mathscr{C}$ is some fixed Reedy category.

\begin{notation}
    Write $X^+$ for the restriction of $X : \mathscr{C} \rightarrow \textbf{sSet}$ along the inclusion $\mathscr{C}^+ \subseteq \mathscr{C}$. Similarly, write $\gamma^+ : X^+ \Rightarrow Y^+$ for the whiskering of a natural transformation $\gamma : X \Rightarrow Y$ between such functors with this inclusion.
\end{notation}

We first construct $R, \kappa, \alpha$ and $H$ for objects of degree $0$:

\begin{definition} \label{def:reedy:inductive0}
    Suppose $X : \mathscr{C}^{\leq 0} \rightarrow \textbf{sSet}$ is levelwise fibrant. Let $c \in \mathscr{C}^{\leq 0}$. Define:
    \begin{itemize}
        \item $R^{\leq 0}(X)(c) := X(c)$;

        \item $(\kappa^{\leq 0}_{X})(c) : X(c) \rightarrow R^{\leq 0}(X)(c)$ to be the identity;

        \item $(\alpha^{\leq 0}_{X})(c) : R^{\leq 0}(X)^+(c) \rightarrow X^+(c)$ to be the identity;

        \item $(H^{\leq 0}_X)(c) : (\kappa_X^{\leq 0})^+(c) \circ (\alpha_X^{\leq 0})(c) \sim 1_{R^{\leq 0}(X)^+(c)}$ to be the trivial homotopy.
    \end{itemize}
    These together define
    $$
        R^{\leq 0} := 1_{(\textbf{sSet}^{\mathscr{C}^{\leq 0}})_{proj}} : (\textbf{sSet}^{\mathscr{C}^{\leq 0}})_{proj} \rightarrow (\textbf{sSet}^{\mathscr{C}^{\leq 0}})_{proj}
    $$
    and $\kappa_X^{\leq 0} : id \Rightarrow R^{\leq 0}$, $\alpha_X^{\leq 0} : (R^{\leq 0})^+ \Rightarrow i^\ast$ and $H_X^{\leq 0} : (\kappa_X^{\leq 0})^+ \alpha_X^{\leq 0} \sim 1_{R^{\leq 0}(X)^+}$.
\end{definition}

It is clear that these are all valid functors and natural transformations, as $\mathscr{C}^{\leq 0}$ is discrete. Moreover, $\alpha^{\leq 0}_X (\kappa^{\leq 0}_X)^+ = id$ by definition.

We now proceed to the inductive case. Before all else, a few constructions must be established:

\begin{notation}\label{prop:reedy:match_kappa}
    Suppose $X : \mathscr{C}^{\leq n} \rightarrow \textbf{sSet}$ is levelwise fibrant, where $n > 0$. Suppose that, writing $X$ to also mean $X^{\leq n - 1} : \mathscr{C}^{\leq n - 1} \rightarrow \textbf{sSet}$, the functor and maps
    \begin{align*}
        R^{\leq n-1}(X) &: \mathscr{C}^{\leq n - 1} \rightarrow \textbf{sSet} \\
        \kappa^{\leq n-1}_X &: X^{\leq n - 1} \Rightarrow R^{\leq n - 1}(X) \\
        \alpha^{\leq n-1}_X &: R^{\leq n - 1}(X)^+ \Rightarrow (X^+)^{\leq n - 1} \\
        H^{\leq n-1}_X &: (\kappa^{\leq n - 1}_X)^+ \alpha_X^{\leq n - 1} \sim 1_{R^{\leq n - 1}(X)^+}
    \end{align*}
    are all defined to be natural in $X$ and such that $\alpha^{\leq n - 1}_X(\kappa_X^{\leq n - 1})^+ = 1_{(X^+)^{\leq n - 1}}$. Suppose $c \in \mathscr{C}$ such that $\textbf{deg}(c) = n$.

    Then write
    \begin{align*}
        M_c \kappa_X^{\leq n - 1} &: M_c X = M_c X^{\leq n - 1} \rightarrow M_c R^{\leq n - 1}(X) \\
        L_c \kappa_X^{\leq n - 1} &: L_c X = L_c X^{\leq n - 1} \rightarrow L_c R^{\leq n - 1}(X)
    \end{align*}
    for the natural maps defined via $\kappa^{\leq n - 1}_X$.
\end{notation}

Note that this notation is not just an application of the functors $M_c, L_c : \textbf{sSet}^{\mathscr{C}^{\leq n}} \rightarrow \textbf{sSet}$, as $R^{\leq n - 1}(X)$ and $\kappa_X^{\leq n - 1}$ are not defined on all of $\mathscr{C}^{\leq n}$. Indeed, we find that $M_c$ and $L_c$ can be instead defined on the domain category $\textbf{sSet}^{\mathscr{C}^{\leq n - 1}}$.

We will need to be more careful with inducing maps on latching objects via $\alpha$ and $H$, as these are moreover only defined on $(\mathscr{C}^+)^{\leq n - 1}$. We will write $\mathscr{L}_c$ instead of $L_c$ in these cases, to remind ourselves there is no map $R^{\leq n - 1}(X) \rightarrow X^{\leq n - 1}$ nor homotopy $R^{\leq n - 1}(X) \rightarrow R^{\leq n - 1}(X)^{\textbf{nerve}(I[1])}$ we are starting with.

\begin{definition}
    Suppose $\mathscr{C}$ is a Reedy category with $c \in \mathscr{C}$ such that $\textbf{deg}(c) = n$. Then define
    $$
        \mathscr{L}_c : \textbf{sSet}^{(\mathscr{C}^+)^{\leq n - 1}} \rightarrow \textbf{sSet}
    $$
    to send $X \mapsto \text{colim}_{\partial (\mathscr{C}^+ \downarrow c)} X$.
\end{definition}

This is indeed well-defined, as the categories $\partial (\mathscr{C}^+ \downarrow c)$ are entirely contained within $(\mathscr{C}^+)^{\leq n - 1}$.

\begin{proposition}
    Suppose $\mathscr{C}$ is a Reedy category with $c \in \mathscr{C}$ such that $\textbf{deg}(c) = n$. Then the diagram
    \[\begin{tikzcd}
    	{\textbf{sSet}^{\mathscr{C}^{\leq n}}} & {\textbf{sSet}} \\
    	{\textbf{sSet}^{(\mathscr{C}^{\leq n-1})^+}}
    	\arrow["{L_c}", from=1-1, to=1-2]
    	\arrow[from=1-1, to=2-1]
    	\arrow["{\mathscr{L}_c}"', from=2-1, to=1-2]
    \end{tikzcd}\]
    commutes.
\end{proposition}

\begin{proof}
    This is by definition.
\end{proof}

We immediately obtain the following:

\begin{proposition}\label{prop: H_X^n}
    Assume the conditions of Notation \ref{prop:reedy:match_kappa}. Then there is a natural map
    $$
        \mathscr{L}_c \alpha^{\leq n - 1}_X : L_c R^{\leq n - 1}(X) \rightarrow L_c X^{\leq n - 1} = L_c X
    $$
    induced by $\alpha_X^{\leq n - 1}$.
\end{proposition}

\begin{proposition}
    Assume the conditions of Notation \ref{prop:reedy:match_kappa}. Then there is a natural map
    $$
        \mathscr{L}_c H_X^{\leq n - 1} : L_c R^{\leq n - 1}(X) \rightarrow L_c (R^{\leq n - 1}(X)^{\textbf{nerve}(I[1])})
    $$
    induced by $H_X^{\leq n - 1}$.
\end{proposition}

\begin{proof}
    $H^{\leq n - 1}$ is defined as a natural transformation of functors 
    \[\begin{tikzcd}
    	{ (\textbf{sSet}^{(\mathscr{C}^{\leq n-1})^+})_{proj}} &&& {(\textbf{sSet}^{(\mathscr{C}^{\leq n-1})^+})_{proj}}
    	\arrow[""{name=0, anchor=center, inner sep=0}, "{((R^{\leq n - 1})^+)^{\textbf{nerve}(I[1])}}"', curve={height=18pt}, from=1-1, to=1-4]
    	\arrow[""{name=1, anchor=center, inner sep=0}, "{(R^{\leq n - 1})^+}", curve={height=-18pt}, from=1-1, to=1-4]
    	\arrow["{H^{\leq n - 1}}"', shorten <=5pt, shorten >=5pt, Rightarrow, from=1, to=0]
    \end{tikzcd}\]
    and so is sufficient to define a map on latching objects as stated.
\end{proof}

\begin{proposition}
    Assume the conditions of Notation \ref{prop:reedy:match_kappa}. Then
    \begin{align*}
        L_c \kappa^{\leq n - 1}_X \circ \mathscr{L}_c \alpha^{\leq n - 1}_X &= \mathscr{L}_c \big( (\kappa^{\leq n - 1}_X)^+ \alpha^{\leq n - 1}_X \big) \\
        \mathscr{L}_c \alpha^{\leq n - 1}_X \circ  L_c \kappa^{\leq n - 1}_X &= \mathscr{L}_c \big( \alpha^{\leq n - 1}_X (\kappa^{\leq n - 1}_X)^+ \big).
    \end{align*}
\end{proposition}

\begin{proof}
    We have that $L_c \kappa^{\leq n - 1}_X = \mathscr{L}_c (\kappa^{\leq n - 1}_X)^+$. The result then follows.
\end{proof}

We will need a particular homotopy induced by these maps. Note importantly that since $M_c (R^{\leq n - 1}(X)^{\textbf{nerve}(I[1])})$ is defined as a limit, we have a natural isomorphism
$$
    M_c (R^{\leq n - 1}(X)^{\textbf{nerve}(I[1])}) \cong M_c (R^{\leq n - 1}(X))^{\textbf{nerve}(I[1])}.
$$

\begin{notation}
    Assume the conditions of Notation \ref{prop:reedy:match_kappa}. Then write
    \begin{align*}
        L_c(s_{R^{\leq n-1}(X)}), L_c(t_{R^{\leq n-1}(X)}) : L_c(R^{\leq n-1}(X)^{\textbf{nerve}(I[1])}) \rightarrow L_c(R^{\leq n-1}(X)) \\
        M_c(s_{R^{\leq n-1}(X)}), M_c(t_{R^{\leq n-1}(X)}) : M_c(R^{\leq n-1}(X)^{\textbf{nerve}(I[1])}) \rightarrow M_c(R^{\leq n-1}(X))
    \end{align*}
    for the maps induced by $s_{R^{\leq n-1}(X)}$ and $t_{R^{\leq n-1}(X)}$, respectively.
\end{notation}

\begin{proposition} \label{prop:reedyfib:diag_latchmatch_comms}
    Assume the conditions of Notation \ref{prop:reedy:match_kappa}. Then, letting $g_0, g_1, g_2, g_3$ be latching-matching maps for appropriate functors, the diagram
    \[\hskip-1cm\begin{tikzcd}
    	& {L_c X} && {L_c R^{\leq n-1}(X)} && {L_c R^{\leq n-1}(X)} \\
    	{M_c X} & {L_c R^{\leq n-1}(X)} && {L_c(R^{\leq n-1}(X)^{\textbf{nerve}(I[1])})} && {L_c R^{\leq n-1}(X)} \\
    	& {M_c(R^{\leq n-1}(X))} && {M_c(R^{\leq n-1}(X)^{\textbf{nerve}(I[1])})} && {M_c(R^{\leq n-1}(X))} \\
    	& {M_c(R^{\leq n-1}(X))} && {(M_c R^{\leq n-1}(X))^{\textbf{nerve}(I[1])}} && {M_c(R^{\leq n-1}(X))}
    	\arrow["{g_0}"', from=1-2, to=2-1]
    	\arrow["{L_c \kappa^{\leq n - 1}_X}", from=1-2, to=2-2]
    	\arrow["{\mathscr{L} \alpha^{\leq n - 1}_X}"', from=1-4, to=1-2]
    	\arrow["id", from=1-4, to=1-6]
    	\arrow["{\mathscr{L}_c H_X^{\leq n-1}}"', from=1-4, to=2-4]
    	\arrow["id", from=1-6, to=2-6]
    	\arrow["{M_c \kappa^{\leq n - 1}_X}"', from=2-1, to=3-2]
    	\arrow["{g_1}", from=2-2, to=3-2]
    	\arrow["{L_c(s_{R^{\leq n-1}(X)})}"', from=2-4, to=2-2]
    	\arrow["{L_c(t_{R^{\leq n-1}(X)})}", from=2-4, to=2-6]
    	\arrow["{g_2}"', from=2-4, to=3-4]
    	\arrow["{g_3}", from=2-6, to=3-6]
    	\arrow["id"', from=3-2, to=4-2]
    	\arrow["{M_c(s_{R^{\leq n-1}(X)})}", from=3-4, to=3-2]
    	\arrow["{M_c(t_{R^{\leq n-1}(X)})}"', from=3-4, to=3-6]
    	\arrow["\cong"', from=3-4, to=4-4]
    	\arrow["id", from=3-6, to=4-6]
    	\arrow["{s_{M_c R^{\leq n-1}(X)}}", from=4-4, to=4-2]
    	\arrow["{t_{M_c R^{\leq n-1}(X)}}"', from=4-4, to=4-6]
    \end{tikzcd}\]
    commutes. Moreover, this diagram is natural in $X$.
\end{proposition}

\begin{proof}
    This is a matter of checking definitions.
\end{proof}

We are now ready to commence the inductive step to constructing our Reedy fibrant replacement functor:

\begin{definition} \label{def:reedy:r_constr}
    Assume the conditions of Notation \ref{prop:reedy:match_kappa}. Then define $R^{\leq n}|_{\mathscr{C}^{\leq n-1}} := R^{\leq n - 1}$. This is thus sufficient to build $M_c R^{\leq n}(X)$ and a map $M_c X \rightarrow M_c R^{\leq n}(X)$ via $ M_c \kappa_X^{\leq n - 1}$. Hence, define $R^{\leq n}(X)(c)$ to be the pullback
    \[\begin{tikzcd}
    	{R^{\leq n}(X)(c)} && {X(c)} \\
    	&& {M_c X} \\
    	{(M_c R^{\leq n}(X))^{\textbf{nerve}(I[1])}} && {M_c R^{\leq n}(X)}
    	\arrow[from=1-1, to=1-3]
    	\arrow[from=1-1, to=3-1]
    	\arrow["\ulcorner"{anchor=center, pos=0.125}, draw=none, from=1-1, to=3-3]
    	\arrow[from=1-3, to=2-3]
    	\arrow["{M_c \kappa_X^{\leq n - 1}}", from=2-3, to=3-3]
    	\arrow["{s_{M_c R^{\leq n}(X)}}"', from=3-1, to=3-3]
    \end{tikzcd}\]
    Moreover, we have a natural map
    $$
        R^{\leq n}(X)(c) \rightarrow (M_c R^{\leq n}(X))^{\textbf{nerve}(I[1])} \xrightarrow{t_{M_c R^{\leq n}(X)}} M_c R^{\leq n}(X)
    $$
    where the first map is pullback projection. Note moreover that $L_c R^{\leq n}(X)$ is already defined. Then, consider the diagram
    \[\begin{tikzcd}[sep=small]
    	{L_c R^{\leq n}(X)} && {L_c X} \\
    	& {X(c) \times_{M_c R^{\leq n}(X)} (M_c R^{\leq n}(X))^{\textbf{nerve}(I[1])}} & {X(c)} \\
    	{L_c (R^{\leq n}(X)^{\textbf{nerve}(I[1])})} && {M_c X} \\
    	{M_c (R^{\leq n}(X)^{\textbf{nerve}(I[1])})} & {(M_c R^{\leq n}(X))^{\textbf{nerve}(I[1])}} & {M_c R^{\leq n}(X)}
    	\arrow["{\mathscr{L}_c \alpha^{\leq n -1}}", from=1-1, to=1-3]
    	\arrow[dashed, from=1-1, to=2-2]
    	\arrow["{\mathscr{L}_c H^{\leq n - 1}_X}"', from=1-1, to=3-1]
    	\arrow[from=1-3, to=2-3]
    	\arrow[from=2-2, to=2-3]
    	\arrow[from=2-2, to=4-2]
    	\arrow[from=2-3, to=3-3]
    	\arrow[from=3-1, to=4-1]
    	\arrow["{M_c \kappa^{\leq n - 1}_X}", from=3-3, to=4-3]
    	\arrow["\cong", from=4-1, to=4-2]
    	\arrow["{s_{M_c R^{\leq n}(X)}}", from=4-2, to=4-3]
    \end{tikzcd}\]
    This commutes by Proposition \ref{prop:reedyfib:diag_latchmatch_comms}, so yields a map from the latching object. As the map
    \begin{align*}
        L_c R^{\leq n}(X) \xrightarrow{\mathscr{L}_c H_X^{\leq n-1}} L_c(R^{\leq n}(X)^{\textbf{nerve}(I[1])})
        &\rightarrow M_c(R^{\leq n}(X)^{\textbf{nerve}(I[1])}) \\
        &\rightarrow (M_c R^{\leq n}(X))^{\textbf{nerve}(I[1])} \\
        &\xrightarrow{t_{M_c R^{\leq n}(X)}} M_c R^{\leq n}(X)
    \end{align*}
    is precisely the latching-matching map by Proposition \ref{prop:reedyfib:diag_latchmatch_comms}, we have that the composition
    $$
        L_c R^{\leq n}(X) \rightarrow R^{\leq n}(X)(c) \rightarrow M_c R^{\leq n}(X)
    $$
    is precisely the latching-matching map. This is thus sufficient to define $R^{\leq n}(X)$.

    A natural definition of $R^{\leq n}(f)$ for $f : X \rightarrow Y$ results in the functor
    $$
        R^{\leq n} : (\textbf{sSet}^{\mathscr{C}^{\leq n}})_{proj} \rightarrow (\textbf{sSet}^{\mathscr{C}^{\leq n}})_{proj}.
    $$
\end{definition}

The fact that $R^{\leq n}$ is functorial is due to Proposition \ref{prop:reedyfib:diag_latchmatch_comms}, in particular the fact that the diagram in question is natural in $X$. This allows us to construct morphisms $R^{\leq n}(f)$ by the unsual inductive procedure to obtain natural transformations between functors $\mathscr{C}^{\leq n} \rightarrow \textbf{sSet}$, discussed for instance in \cite[pg. 287]{hirschhornModelCategoriesTheir2009}.

Before we proceed further, note that we will not be able to prove our inductive definition of $\kappa_X^{\leq n}$ is a natural transformation until we have established the corresponding inductive definition of $H_X^{\leq n}$. However, we are able to define its levels without yet having naturality:

\begin{definition} \label{def:reedy:kappa_inductive}
    Assume the conditions of Notation \ref{prop:reedy:match_kappa}. Then define $(\kappa^{\leq n}_X)(c) : X(c) \rightarrow R^{\leq n}(X)(c)$ by the identity on $X(c)$ together with the map
    $$
        X(c) \rightarrow M_c X \xrightarrow{M_c \kappa_X^{\leq n - 1}} M_c R^{\leq n}(X) \xrightarrow{c_{M_c R^{\leq n}(X)}} (M_c R^{\leq n}(X))^{\textbf{nerve}(I[1])}.
    $$
    Moreover, define $(\kappa_X^{\leq n})(d) := (\kappa_X^{\leq n-1})(d)$ for $\textbf{deg}(d) \leq n - 1$.
\end{definition}

\begin{definition}
    Assume the conditions of Notation \ref{prop:reedy:match_kappa}. Then define $(\alpha^{\leq n}_X)(c) : R^{\leq n}(X)(c) \rightarrow X(c)$ by the pullback projection to $X(c)$. Moreover, define $(\alpha_X^{\leq n})(d) := (\alpha_X^{\leq n-1})(d)$ for $\textbf{deg}(d) \leq n - 1$.
\end{definition}

Some explanation is likely needed for why $\alpha_X$ is only defined on $\mathscr{C}^+$ rather than all of $\mathscr{C}$. We should not expect a retract $R(X) \rightarrow X$ of $\kappa_X$ for all of $X$, as such a map would imply that the matching maps $X(c) \rightarrow M_c X$ would be retracts of the maps $R(X)(c) \rightarrow M_c R(X)$. Indeed, we would have diagrams of the form
\[\begin{tikzcd}
	{X(c)} && {R(X)(c)} && {X(c)} \\
	{M_c X} && {M_c R(X)} && {M_c X}
	\arrow["{\kappa_X(c)}"', from=1-1, to=1-3]
	\arrow["id", curve={height=-18pt}, from=1-1, to=1-5]
	\arrow[from=1-1, to=2-1]
	\arrow["{\alpha_X(c)}"', from=1-3, to=1-5]
	\arrow[from=1-3, to=2-3]
	\arrow[from=1-5, to=2-5]
	\arrow["{M_c \kappa_X}", from=2-1, to=2-3]
	\arrow["id"', curve={height=18pt}, from=2-1, to=2-5]
	\arrow["{M_c \alpha_X}", from=2-3, to=2-5]
\end{tikzcd}\]
As fibrations are closed under retract, this would imply that $X$ was already Reedy fibrant, which is not generally the case. Thus, it must not be possible in general to define $M_c \alpha_X$ and therefore to define $\alpha_X$ on $\mathscr{C}^-$. Regardless, it is perfectly reasonable to define $\alpha_X$ on $\mathscr{C}^+$ as we will prove later, which is more than sufficient for our purposes.

We now turn to defining our homotopies. Note that $(R(X)^{\leq n}(c))^{\textbf{nerve}(I[1])}$ may alternatively be written as
$$
     X(c)^{\textbf{nerve}(I[1])} \times_{(M_c R^{\leq n}(X))^{\textbf{nerve}(I[1])}} \big((M_c R^{\leq n}(X))^{\textbf{nerve}(I[1])} \big)^{\textbf{nerve}(I[1])}.
$$

\begin{definition}
    Assume the conditions of Notation \ref{prop:reedy:match_kappa}. Then define $(H^{\leq n}_X)(c) : R^{\leq n}(X)(c) \rightarrow (R^{\leq n}(X)(c))^{\textbf{nerve}(I[1])}$ to be induced by the map of cospans
    \[\begin{tikzcd}
    	{(M_c R^{\leq n}(X))^{\textbf{nerve}(I[1])}} & {M_c R^{\leq n}(X)} & {X(c)} \\
    	{\big((M_c R^{\leq n}(X))^{\textbf{nerve}(I[1])}\big)^{\textbf{nerve}(I[1])}} & {(M_c R^{\leq n}(X))^{\textbf{nerve}(I[1])}} & {X(c)^{\textbf{nerve}(I[1])}}
    	\arrow[from=1-1, to=1-2]
    	\arrow["{D_{M_c R^{\leq n}(X)}}"', from=1-1, to=2-1]
    	\arrow["{c_{M_c R^{\leq n}(X)}}"', from=1-2, to=2-2]
    	\arrow[from=1-3, to=1-2]
    	\arrow["{c_{X(c)}}", from=1-3, to=2-3]
    	\arrow[from=2-1, to=2-2]
    	\arrow[from=2-3, to=2-2]
    \end{tikzcd}\] 
    Moreover, define $(H_X^{\leq n})(d) := (H_X^{\leq n-1})(d)$ for $\textbf{deg}(d) \leq n - 1$.
\end{definition}

We now must prove that $\kappa^{\leq n}, \alpha^{\leq n}$ and $H^{\leq n}$ are natural in $\mathscr{C}$ and in $(\textbf{sSet}^{\mathscr{C}})_{proj}$. Before we do so, a result is needed:

\begin{proposition} \label{prop:reedy:hkappa_const_htpy}
    Assume the conditions of Notation \ref{prop:reedy:match_kappa}. Then the map $(H_X^{\leq n})(c) \circ (\kappa_X^{\leq n})^+(c)$ is the constant homotopy on $(\kappa_X^{\leq n})^+(c)$.
\end{proposition}

\begin{proof}
    We have that this map is defined by the pullback of the two composite maps
    $$
        X(c) \xrightarrow{(\kappa^{\leq n}_X)^+(c)} X(c) \times_{M_c R^{\leq n}(X)} (M_c R^{\leq n}(X))^{\textbf{nerve}(I[1])} \rightarrow X(c) \xrightarrow{c_{X(c)}} X(c)^{\textbf{nerve}(I[1])}
    $$
    and
    \begin{align*}
        X(c) &\xrightarrow{(\kappa^{\leq n}_X)^+(c)} X(c) \times_{M_c R^{\leq n}(X)} (M_c R^{\leq n}(X))^{\textbf{nerve}(I[1])} \\
        &\rightarrow (M_c R^{\leq n}(X))^{\textbf{nerve}(I[1])} \\
        &\xrightarrow{D_{M_c R^{\leq n}(X)}} \big((M_c R^{\leq n}(X))^{\textbf{nerve}(I[1])}\big)^{\textbf{nerve}(I[1])}.
    \end{align*}
    The former of these maps is evidently just $c_{X(c)}$. The latter is then the map
    \begin{align*}
        X(c) \rightarrow M_c X \xrightarrow{M_c \kappa_X^{\leq n - 1}} M_c R^{\leq n}(X) &\xrightarrow{c_{M_c R^{\leq n}(X)}} (M_c R^{\leq n}(X))^{\textbf{nerve}(I[1])} \\
        &\xrightarrow{D_{M_c R^{\leq n}(X)}} \big((M_c R^{\leq n}(X))^{\textbf{nerve}(I[1])}\big)^{\textbf{nerve}(I[1])}        
    \end{align*}
    which by Proposition \ref{prop:reedy:db_precomp_const} is the map
    \begin{align*}
        X(c) \rightarrow M_c X \rightarrow M_c R^{\leq n}(X) &\xrightarrow{c_{M_c R^{\leq n}(X)}} (M_c R^{\leq n}(X))^{\textbf{nerve}(I[1])} \\
        &\xrightarrow{c_{(M_c R^{\leq n}(X))^{\textbf{nerve}(I[1])}}} \big((M_c R^{\leq n}(X))^{\textbf{nerve}(I[1])}\big)^{\textbf{nerve}(I[1])}.   
    \end{align*}
    Thus, the composite $(H_X^{\leq n})(c) \circ (\kappa_X^{\leq n})^+(c)$ is the constant homotopy on $(\kappa_X^{\leq n})^+(c)$, as needed.
\end{proof}

\begin{proposition}\label{prop: ext of kappa}
    Assume the conditions of Notation \ref{prop:reedy:match_kappa}. Then $\kappa^{\leq n}_X : X \rightarrow R^{\leq n}(X)$ is a valid morphism.
\end{proposition}

\begin{proof}
    It suffices to show that the diagram
    \[\begin{tikzcd}
    	{L_c X^{\leq n}} & {X(c)} & {M_c X} \\
    	{L_c R^{\leq n}(X)} & {R^{\leq n}(X)(c)} & {M_c R^{\leq n}(X)}
    	\arrow[from=1-1, to=1-2]
    	\arrow["{L_c \kappa^{\leq n}_X}"', from=1-1, to=2-1]
    	\arrow[from=1-2, to=1-3]
    	\arrow["{(\kappa^{\leq n}_X)(c)}"', from=1-2, to=2-2]
    	\arrow["{M_c \kappa^{\leq n}_X}", from=1-3, to=2-3]
    	\arrow[from=2-1, to=2-2]
    	\arrow[from=2-2, to=2-3]
    \end{tikzcd}\]
    commutes. It is immediate that the right-hand square commutes. For the left-hand side, it suffices to show that the two diagrams
    \[\begin{tikzcd}
    	{L_c X} & {L_c R^{\leq n}(X)} & {L_c X} \\
    	{X(c)} && {X(c)}
    	\arrow["{L_c \kappa^{\leq n}_X}", from=1-1, to=1-2]
    	\arrow[from=1-1, to=2-1]
    	\arrow["{\mathscr{L}_c \alpha^{\leq n}_X}", from=1-2, to=1-3]
    	\arrow[from=1-3, to=2-3]
    	\arrow["id", from=2-1, to=2-3]
    \end{tikzcd}\]
    and
    \[\begin{tikzcd}
    	{L_c X} & {L_c R^{\leq n}(X)} & {L_c (R^{\leq n}(X)^{\textbf{nerve}(I[1])})} & {M_c (R^{\leq n}(X)^{\textbf{nerve}(I[1])})} \\
    	{X(c)} & {M_c X} & {M_c R^{\leq n}(X)} & {(M_c R^{\leq n}(X))^{\textbf{nerve}(I[1])}}
    	\arrow["{L_c \kappa_X^{\leq n}}", from=1-1, to=1-2]
    	\arrow[from=1-1, to=2-1]
    	\arrow["{\mathscr{L}_c H_X^{\leq n}}", from=1-2, to=1-3]
    	\arrow[from=1-3, to=1-4]
    	\arrow["\cong", from=1-4, to=2-4]
    	\arrow[from=2-1, to=2-2]
    	\arrow["{M_c \kappa_X^{\leq n}}"', from=2-2, to=2-3]
    	\arrow[from=2-3, to=2-4]
    \end{tikzcd}\]
    commute. The former commutes because $\alpha^{\leq n-1}_X (\kappa^{\leq n-1}_X)^+ = 1_{(X^+)^{\leq n-1}}$ by our assumptions, while the latter commutes because
    $$
        H_X^{\leq n-1} (\kappa^{\leq n-1}_X)^+ : (\kappa^{\leq n-1}_X)^+ \sim (\kappa^{\leq n-1}_X)^+
    $$
    is levelwise the trivial homotopy by Proposition \ref{prop:reedy:hkappa_const_htpy}.
\end{proof}

\begin{proposition}
    Assume the conditions of Notation \ref{prop:reedy:match_kappa}. Then $\alpha_X^{\leq n} : R^{\leq n}(X)^+ \rightarrow (X^{\leq n})^+$ is a valid morphism.
\end{proposition}

\begin{proof}
    Since $(\mathscr{C}^+)^{\leq n}$ exhibits trivial matching objects, defining such a functor requires only a commutative diagram
    \[\begin{tikzcd}
    	{L_c R(X)^{\leq n}} && {R(X)^{\leq n}(c)} \\
    	{L_c X^{\leq n}} && {X^{\leq n}(c)}
    	\arrow[from=1-1, to=1-3]
    	\arrow["{\mathscr{L}_c \alpha^{\leq n}_X}"', from=1-1, to=2-1]
    	\arrow["{\alpha^{\leq n}_X(c)}", from=1-3, to=2-3]
    	\arrow[from=2-1, to=2-3]
    \end{tikzcd}\]
    Note however that by the definition of the uppermost horizontal map, this commutativity is trivial.
\end{proof}

To prove naturality of $H_X^{\leq n}$, we will need an intermediate result:

\begin{proposition} \label{prop:reedy:hd=hh}
    Assume the conditions of Notation \ref{prop:reedy:match_kappa}. Then
    $$
        (H^{\leq n}_X)(c)^{\textbf{nerve}(I[1])} \circ (H^{\leq n}_X)(c) = D_{R^{\leq n}(X)(c)} \circ (H^{\leq n}_X)(c).
    $$
\end{proposition}

\begin{proof}
    Unwinding definitions, we have that $(H^{\leq n}_X)(c)^{\textbf{nerve}(I[1])} \circ (H^{\leq n}_X)(c)$ is given by the map of cospans
    \[\begin{tikzcd}
    	{(M_c R^{\leq n}(X))^{\textbf{nerve}(I[1])}} & {M_c R^{\leq n}(X)} & {X(c)} \\
    	{(M_c R^{\leq n}(X))^{\textbf{nerve}(I[1] \times I[1])}} & {(M_c R^{\leq n}(X))^{\textbf{nerve}(I[1])}} & {X(c)^{\textbf{nerve}(I[1])}} \\
    	{(M_c R^{\leq n}(X))^{\textbf{nerve}(I[1]\times I[1] \times I[1])}} & {(M_c R^{\leq n}(X))^{\textbf{nerve}(I[1] \times I[1])}} & {X(c)^{\textbf{nerve}(I[1] \times I[1])}}
    	\arrow[from=1-1, to=1-2]
    	\arrow["{D_{M_c R^{\leq n}(X)}}"', from=1-1, to=2-1]
    	\arrow["{c_{M_c R^{\leq n}(X)}}"', from=1-2, to=2-2]
    	\arrow[from=1-3, to=1-2]
    	\arrow["{c_{X(c)}}", from=1-3, to=2-3]
    	\arrow[from=2-1, to=2-2]
    	\arrow["{D_{(M_c R^{\leq n}(X))}^{\textbf{nerve}(I[1])}}"', from=2-1, to=3-1]
    	\arrow["{c_{(M_c R^{\leq n}(X))}^{\textbf{nerve}(I[1])}}"', from=2-2, to=3-2]
    	\arrow[from=2-3, to=2-2]
    	\arrow["{c_{X(c)}^{\textbf{nerve}(I[1])}}", from=2-3, to=3-3]
    	\arrow[from=3-1, to=3-2]
    	\arrow[from=3-3, to=3-2]
    \end{tikzcd}\]
    so that the three vertical composites are given by the map $I[1]^3 \rightarrow I[1]$ sending $(i, j, k) \mapsto ijk$ and the map $I[1]^2 \rightarrow \ast$. We then see that the map $D_{R^{\leq n}(X)(c)} \circ (H^{\leq n}_X)(c)$ is given in turn by the composite
    \[\begin{tikzcd}
    	{(M_c R^{\leq n}(X))^{\textbf{nerve}(I[1])}} & {M_c R^{\leq n}(X)} & {X(c)} \\
    	{(M_c R^{\leq n}(X))^{\textbf{nerve}(I[1] \times I[1])}} & {(M_c R^{\leq n}(X))^{\textbf{nerve}(I[1])}} & {X(c)^{\textbf{nerve}(I[1])}} \\
    	{(M_c R^{\leq n}(X))^{\textbf{nerve}(I[1]\times I[1] \times I[1])}} & {(M_c R^{\leq n}(X))^{\textbf{nerve}(I[1] \times I[1])}} & {X(c)^{\textbf{nerve}(I[1] \times I[1])}}
    	\arrow[from=1-1, to=1-2]
    	\arrow["{D_{M_c R^{\leq n}(X)}}"', from=1-1, to=2-1]
    	\arrow["{c_{M_c R^{\leq n}(X)}}"', from=1-2, to=2-2]
    	\arrow[from=1-3, to=1-2]
    	\arrow["{c_{X(c)}}", from=1-3, to=2-3]
    	\arrow[from=2-1, to=2-2]
    	\arrow["{D_{(M_c R^{\leq n}(X))^{\textbf{nerve}(I[1])}}}"', from=2-1, to=3-1]
    	\arrow["{D_{M_c R^{\leq n}(X)}}"', from=2-2, to=3-2]
    	\arrow[from=2-3, to=2-2]
    	\arrow["{D_{X(c)}}", from=2-3, to=3-3]
    	\arrow[from=3-1, to=3-2]
    	\arrow[from=3-3, to=3-2]
    \end{tikzcd}\]
    which is given by precisely the same maps, as needed.
\end{proof}

\begin{proposition}
    Assume the conditions of Notation \ref{prop:reedy:match_kappa}. Then $H^{\leq n}_X : R^{\leq n}(X)^+ \rightarrow (R^{\leq n}(X)^+)^{\textbf{nerve}(I[1])}$ is a valid morphism.
\end{proposition}

\begin{proof}
    Again, it suffices to show that the diagram
    \[\begin{tikzcd}
    	{L_c R^{\leq n}(X)} && {R^{\leq n}(X)(c)} \\
    	{L_c(R^{\leq n}(X)^{\textbf{nerve}(I[1])})} && {R^{\leq n}(X)(c)^{\textbf{nerve}(I[1])}}
    	\arrow[from=1-1, to=1-3]
    	\arrow["{\mathscr{L}_c H^{\leq n}_X}"', from=1-1, to=2-1]
    	\arrow["{H^{\leq n}_X(c)}", from=1-3, to=2-3]
    	\arrow[from=2-1, to=2-3]
    \end{tikzcd}\]
    commutes.

    We subdivide this question by splitting up $(R^{\leq n}(X)(c))^{\textbf{nerve}(I[1])}$ into $X(c)^{\textbf{nerve}(I[1])}$ and $\big((M_c R^{\leq n}(X))^{\textbf{nerve}(I[1])} \big)^{\textbf{nerve}(I[1])}$. It suffices then to check that the maps into these two objects induce commutative diagrams.
    
    One may then first check that the diagram
    \[\begin{tikzcd}
    	{L_c R^{\leq n}(X)} && {R^{\leq n}(X)(c)} \\
    	{L_c(R^{\leq n}(X)^{\textbf{nerve}(I[1])})} & {R^{\leq n}(X)(c)^{\textbf{nerve}(I[1])}} & {X(c)} \\
    	{L_c(X^{\textbf{nerve}(I[1])})} && {X(c)^{\textbf{nerve}(I[1])}}
    	\arrow[from=1-1, to=1-3]
    	\arrow["{\mathscr{L}_c H^{\leq n}_X}"', from=1-1, to=2-1]
    	\arrow["{\alpha^{\leq n}_X(c)}", from=1-3, to=2-3]
    	\arrow[from=2-1, to=2-2]
    	\arrow["{\mathscr{L}_c((\alpha^{\leq n}_X)^{\textbf{nerve}(I[1])})}"', from=2-1, to=3-1]
    	\arrow["{\alpha^{\leq n}_X(c)^{\textbf{nerve}(I[1])}}"{description}, from=2-2, to=3-3]
    	\arrow["{c_{X(c)}}", from=2-3, to=3-3]
    	\arrow[from=3-1, to=3-3]
    \end{tikzcd}\]
    commutes, by unwinding the definition of $H^{\leq n}_X$ inductively and checking first the bottom-left square and then the outermost boundary of the diagram.

    It then suffices to prove that the diagram
    \[\begin{tikzcd}[sep=small]
    	{L_c R^{\leq n}(X)} && {R^{\leq n}(X)(c)} \\
    	{L_c(R^{\leq n}(X)^{\textbf{nerve}(I[1])})} & {L_c(R^{\leq n}(X)^{\textbf{nerve}(I[1])})} & {(M_c R^{\leq n}(X))^{\textbf{nerve}(I[1])}} \\
    	& {L_c\big((R^{\leq n}(X)^{\textbf{nerve}(I[1])})^{\textbf{nerve}(I[1])}\big)} & {\big((M_cR^{\leq n}(X))^{\textbf{nerve}(I[1])}\big)^{\textbf{nerve}(I[1])}}
    	\arrow[from=1-1, to=1-3]
    	\arrow["{\mathscr{L}_c H^{\leq n}_X}"', from=1-1, to=2-1]
    	\arrow["{\mathscr{L}_c H^{\leq n}_X}"{pos=0.6}, from=1-1, to=2-2]
    	\arrow[from=1-3, to=2-3]
    	\arrow["{\mathscr{L}_c ((H^{\leq n}_X)^{\textbf{nerve}(I[1])})}"'{pos=0.1}, from=2-1, to=3-2]
    	\arrow[from=2-2, to=2-3]
    	\arrow["{L_c D_{R^{\leq n}(X)}}", from=2-2, to=3-2]
    	\arrow["{D_{M_c R^{\leq n}(X)}}", from=2-3, to=3-3]
    	\arrow[from=3-2, to=3-3]
    \end{tikzcd}\]
    commutes. By Proposition \ref{prop:reedy:hd=hh}, we have that the leftmost polygon commutes, while the remaining two are by inspection.

    That this diagram suffices is because one may show that the diagram
    \[\begin{tikzcd}
    	{L_c(R^{\leq n}(X)^{\textbf{nerve}(I[1])})} & {R^{\leq n}(X)(c)^{\textbf{nerve}(I[1])}} \\
    	{L_c((R^{\leq n}(X)^{\textbf{nerve}(I[1])})^{\textbf{nerve}(I[1])})} & {M_c((R^{\leq n}(X)^{\textbf{nerve}(I[1])})^{\textbf{nerve}(I[1])})}
    	\arrow[from=1-1, to=1-2]
    	\arrow["{\mathscr{L}_c ((H_X^{\leq n})^{\textbf{nerve}(I[1])})}"', from=1-1, to=2-1]
    	\arrow[two heads, from=1-2, to=2-2]
    	\arrow[from=2-1, to=2-2]
    \end{tikzcd}\]
    commutes.
\end{proof}

\begin{proposition}
    Assume the conditions of Notation \ref{prop:reedy:match_kappa}. Then $H^{\leq n}_X, \kappa^{\leq n}_X$ and $\alpha^{\leq n}_X$ are all natural in $X$.
\end{proposition}

\begin{proof}
    All involved constructions to define these natural transformations are themselves natural in $X$, so the result inductively holds in $n$ by inspection.
\end{proof}

We complete our induction by proving all the necessary results about $\kappa^{\leq n}, \alpha^{\leq n}$ and $H^{\leq n}$ for the next inductive step in $n$.

\begin{proposition}
    Assume the conditions of Notation \ref{prop:reedy:match_kappa}. Then
    $$
        \alpha^{\leq n}_X (\kappa^{\leq n}_X)^+ = 1_{(X^+)^{\leq n}}.
    $$
\end{proposition}

\begin{proof}
    This is by definition.
\end{proof}

\begin{proposition}
    Assume the conditions of Notation \ref{prop:reedy:match_kappa}. Then $(H^{\leq n}_X)(c)$ is a very good right homotopy from $(\kappa^{\leq n}_X)^+(c) \circ (\alpha^{\leq n}_X)(c)$ to $1_{R^{\leq n}(X)^+(c)}$.
\end{proposition}

\begin{proof}
    Consider, for some $c \in \mathscr{C}$, the composite morphism
    $$
        R^{\leq n}(X)(c) \xrightarrow{H^{\leq n}_X(c)} R^{\leq n}(X)(c)^{\textbf{nerve}(I[1])} \rightarrow R^{\leq n}(X)(c)
    $$
    where the final map is either the source or target map. Looking at the components of this map, we have the two maps of the form
    $$
        X(c) \xrightarrow{c_{X(c)}} X(c)^{\textbf{nerve}(I[1])} \rightarrow X(c)
    $$
    which are both identities, along with the two maps of the form
    \begin{align*}
        (M_c R^{\leq n}(X))^{\textbf{nerve}(I[1])}
        &\xrightarrow{D_{M_c R^{\leq n}(X)}} ((M_c R^{\leq n}(X))^{\textbf{nerve}(I[1])})^{\textbf{nerve}(I[1])} \\
        &\rightarrow (M_c R^{\leq n}(X))^{\textbf{nerve}(I[1])}.
    \end{align*}
    If the last map in this chain of morphisms is the source map, then the overall map $R^{\leq n}(X)(c) \rightarrow R^{\leq n}(X)(c)$ is given by the map of cospans
    \[\begin{tikzcd}
    	{X(c)} & {M_c X} && {M_c R^{\leq n}(X)} && {(M_c R^{\leq n}(X))^{\textbf{nerve}(I[1])}} \\
    	{X(c)} & {M_c X} && {M_c R^{\leq n}(X)} && {M_c R^{\leq n}(X)} \\
    	{X(c)} & {M_c X} && {M_c R^{\leq n}(X)} && {(M_c R^{\leq n}(X))^{\textbf{nerve}(I[1])}}
    	\arrow[from=1-1, to=1-2]
    	\arrow["id"', from=1-1, to=2-1]
    	\arrow["{M_c \kappa^{\leq n}_X}", from=1-2, to=1-4]
    	\arrow["id"', from=1-2, to=2-2]
    	\arrow["id", from=1-4, to=2-4]
    	\arrow["{s_{M_c R^{\leq n}(X)}}"', from=1-6, to=1-4]
    	\arrow["{s_{M_c R^{\leq n}(X)}}", from=1-6, to=2-6]
    	\arrow[from=2-1, to=2-2]
    	\arrow["id"', from=2-1, to=3-1]
    	\arrow["{M_c \kappa^{\leq n}_X}", from=2-2, to=2-4]
    	\arrow["id"', from=2-2, to=3-2]
    	\arrow["id", from=2-4, to=3-4]
    	\arrow["id"', from=2-6, to=2-4]
    	\arrow["{c_{M_c R^{\leq n}(X)}}", from=2-6, to=3-6]
    	\arrow[from=3-1, to=3-2]
    	\arrow["{M_c \kappa^{\leq n}_X}"', from=3-2, to=3-4]
    	\arrow["{s_{M_c R^{\leq n}(X)}}", from=3-6, to=3-4]
    \end{tikzcd}\]
    It is clear that the first of these maps is $(\alpha_X^{\leq n})(c)$ by the isomorphism $X(c) \cong X(c) \times_{M_c R^{\leq n}(X)} M_c R^{\leq n}(X)$, so that the induced map of pullbacks $R^{\leq n}(X)(c) \rightarrow R^{\leq n}(X)(c)$ is precisely the result of applying $(\kappa^{\leq n}_X)^+(c) \circ (\alpha^{\leq n}_X)(c)$.
    
    If instead the final map in the composite is the target map, we have that the induced map $R^{\leq n}(X)(c) \rightarrow R^{\leq n}(X)(c)$ the identity. Thus, the homotopy $(H_X^{\leq n})(c)$ is between the correct morphisms as desired.
\end{proof}

With all of these inductive pieces in place, we are able to continue the induction to all $n$ and obtain our construction in full:

\begin{definition}
    Suppose $X : \mathscr{C} \rightarrow \textbf{sSet}$ is levelwise fibrant. Define $R(X)$ such that $R(X)|_{\mathscr{C}^{\leq n}} = R^{\leq n}(X)$. Do similarly for $\kappa_X, \alpha_X$ and $H_X$. This defines $R, \kappa, \alpha$ and $H$ in their entirety.
\end{definition}

This completes the definition of $R$. We must now prove that we have indeed obtained a valid Reedy fibrant replacement:

\begin{proposition}
    $\kappa_X : X \rightarrow R(X)$ is a levelwise trivial cofibration.
\end{proposition}

\begin{proof}
    Consider some $c \in \mathscr{C}$ with $\textbf{deg}(c) = n$. It will suffice to prove that
    $$
        \kappa^{\leq n}_X(c) : X(c) \rightarrow X(c) \times_{M_c R^{\leq n}(X)} (M_c R^{\leq n}(X))^{\textbf{nerve}(I[1])}
    $$
    is a trivial cofibration in $\textbf{sSet}$.

    Note that the existence of $\alpha^{\leq n}_X(c)$ implies that this is a levelwise inclusion, hence a cofibration. Moreover, $H^{\leq n}_X(c)$ exhibits $\kappa^{\leq n}_X(c)$ as a weak equivalence in the model structure in $\textbf{sSet}$; indeed, applying geometric realization to the induced map $R^{\leq n}(X)(c) \times \textbf{nerve}(I[1]) \rightarrow R^{\leq n}(X)(c)$ gives a homotopy
    \begin{align*}
        \lvert R^{\leq n}(X)(c) \rvert \times [0, 1] &\cong \lvert R^{\leq n}(X)(c) \rvert \times \lvert \Delta[1] \rvert \\
        & \rightarrow \lvert R^{\leq n}(X)(c) \rvert \times \lvert \textbf{nerve}(I[1]) \rvert \\
        &\rightarrow \lvert R^{\leq n}(X)(c) \rvert
    \end{align*}
    as geometric realization commutes with products and by the map $\Delta[1] \rightarrow \textbf{nerve}(I[1])$ identifying the morphism $0 \rightarrow 1$ in $I[1]$. This homotopy is from $\lvert \kappa^{\leq n}_X(c) \rvert \circ \lvert \alpha^{\leq n}_X(c) \vert$ to the identity, which together with the fact that $\lvert \alpha^{\leq n}_X(c) \rvert \circ \lvert \kappa^{\leq n}_X(c) \vert = id$ shows we have a weak equivalence of simplicial sets and therefore a trivial cofibration as needed.
\end{proof}

Note by 2-out-of-3 that $\alpha_X$ is also a levelwise weak equivalence. Moreover, note that we have not shown that $\kappa_X$ is in general a Reedy trivial cofibration, only a weak equivalence and levelwise cofibration. However, in the case of \emph{elegant} Reedy categories, we do in fact have a Reedy trivial cofibration.

In order to prove that $R(X)$ is Reedy fibrant, we will need an intermediate result on Reedy model categories:

\begin{proposition}[{\cite[Cor. 15.3.12]{hirschhornModelCategoriesTheir2009}}] \label{prop:reedy:matching_fib_inductive}
    Let $\mathscr{C}$ be a Reedy category and $\mathscr{M}$ a model category. Suppose $X\in \mathscr{M}^{\mathscr{C}}$ is Reedy fibrant and $x \in \mathscr{C}$. Then $M_x X$ is fibrant in $\mathscr{M}$.
\end{proposition}

\begin{proposition}
    Suppose $X$ is projective fibrant. Then $R(X)$ is Reedy fibrant.
\end{proposition}

\begin{proof}
    We prove this by induction. Suppose $c \in \mathscr{C}$ has $\textbf{deg}(c) = 0$. Then $R(X)(c) = X(c)$, so we have Reedy fibrancy at $c$ so that $R^{\leq 0}(X)$ is Reedy fibrant.

    For the inductive step, suppose $R^{\leq n - 1}(X)$ is Reedy fibrant. Then, following the methods in \cite[pg. 293]{hirschhornModelCategoriesTheir2009}, for every $d \in \mathscr{C}$ of degree $n$ we can factorize the map $X(d) \sqcup_{L_d X} L_d R^{\leq n - 1}(X) \rightarrow M_d R^{\leq n - 1}(X)$ into a trivial cofibration followed by a fibration. This yields a Reedy fibrant functor
    $$
        Q(X) : \mathscr{C}^{\leq n} \rightarrow \textbf{sSet}
    $$
    where $Q(X)|_{\mathscr{C}^{\leq n - 1}} = R^{\leq n - 1}(X)$. Then we have that $M_c Q(X) = M_c R^{\leq n - 1}(X)$ is a Kan complex by Proposition \ref{prop:reedy:matching_fib_inductive}. Thus, we have a map $(\kappa_X)(c) : X(c) \rightarrow M_c R(X) = M_c R^{\leq n - 1}(X)
    $ between Kan complexes, which induces the map
    $$
        R(X)(c) = X(c) \times_{M_c R(X)} (M_c R(X))^{\textbf{nerve}(I[1])} \rightarrow M_c R(X)
    $$
    as a fibration by Lemma \ref{lemma:reedy:sset_path_fib}.
\end{proof}

There are a few properties of $R$ which we find worth mentioning:

\begin{proposition}
    $f : X \rightarrow Y$ is a levelwise weak equivalence in $(\textbf{sSet}^\mathscr{C})_{proj}$ if and only if $R(f)$ is a levelwise weak equivalence.
\end{proposition}

\begin{proof}
    For any $c \in \mathscr{C}$, we have a natural diagram
    \[\begin{tikzcd}
    	{X(c)} && {Y(c)} \\
    	{X(c) \times_{M_c X} (M_c X)^{\textbf{nerve}(I[1])}} && {Y(c) \times_{M_c Y} (M_c Y)^{\textbf{nerve}(I[1])}}
    	\arrow["{f(c)}", from=1-1, to=1-3]
    	\arrow["{(\kappa_X)(c)}"', from=1-1, to=2-1]
    	\arrow["{(\kappa_Y)(c)}", from=1-3, to=2-3]
    	\arrow["{R(f)(c)}"', from=2-1, to=2-3]
    \end{tikzcd}\]
    where all but the horizontal maps are weak equivalences. By 2-out-of-3, the result holds.
\end{proof}

\begin{proposition}
    $R$ preserves finite products up to natural isomorphism.
\end{proposition}

\begin{proof}
    It is evident that there is a levelwise isomorphism
    $$
        R(X \times Y)(c) \cong R(X)(c) \times R(Y)(c)
    $$
    and that this commutes with matching object maps. For latching object maps, one may assume by induction that the latching objects of $R(X) \times R(Y)$ and $R(X \times Y)$ are isomorphic. Then one quickly checks that the same latching maps are induced in both instances.
\end{proof}
\section{Taking Examples to the (Homotopy) Limit}

To expand our intuitions for $R$'s behavior, we analyze its use in constructing standard homotopy limits. The simplest non-trivial categories to examine are those Reedy categories $\mathscr{C}$ such that $\mathscr{C}^+$ is discrete. In this case, the latching objects and thus latching maps are all trivial, so we can entirely focus on the matching objects. An interesting consequence of this fact is that the Reedy cofibrations in $\textbf{sSet}^{\mathscr{C}}$ are precisely the levelwise cofibrations, meaning the Reedy and injective model structures agree. Our construction therefore necessarily obtains a model of homotopy limit over such diagrams.

Consider, for instance, the cospan diagram $C := \{x \rightarrow y \leftarrow z\}$, where $\textbf{deg}(x) = \textbf{deg}(z) = 1$ and $\textbf{deg}(y) = 0$. A projective fibrant functor $F : C \rightarrow \textbf{sSet}$ then amounts to a cospan of simplicial sets
$$
    C_x \rightarrow C_y \leftarrow C_z
$$
where $C_x, C_y$ and $C_z$ are all Kan complexes. The Reedy fibrant replacement $R(F)$ is then easily checked to be the cospan
$$
    C_x \times_{C_y} C_y^{\textbf{nerve}(I[1])} \rightarrow C_y \leftarrow C_y^{\textbf{nerve}(I[1])} \times_{C_y} C_z.
$$
If one were to take a pullback of this cospan, one immediately obtains our chosen definition of homotopy pullback $C_x \times^h_{C_y} C_z$ in terms of $\Lambda$.

Another example comes from the diagram for equalizers, namely $E := \{x \rightrightarrows y\}$, with $\textbf{deg}(x) = 1$ and $\textbf{deg}(y) = 0$. Then we have, for a functor $G : E \rightarrow \textbf{sSet}$, that $R(G)$ is the diagram
$$
    E_y \times_{E_x^2} (E_x^2)^{\textbf{nerve}(I[1])} \rightrightarrows E_x.
$$
The limit of this diagram is similarly a reasonable model of homotopy-coherent equalizer, whose $0$-simplices are triples $(p, e, q)$, where $e \in (E_y)_0$ and $p, q : \textbf{nerve}(I[1]) \rightarrow E_y$ are paths such that $p(1) = q(0) = e$ and the image of $p(0)$ under the first map $E_x \rightarrow E_y$ is equal to the image of $q(1)$ under the second map.

For an infinite example, consider the poset category $M := \{ \cdots \rightarrow 3 \rightarrow 2 \rightarrow 1 \rightarrow 0\}$, whose objects are the elements of $\mathbb{Z}_{\geq 0}$ with a morphism $a \rightarrow b$ if and only if $b \leq a$. This is a Reedy category where $\textbf{deg}(x) = x$ for all objects $x \in M$. For a levelwise fibrant functor $J : M \rightarrow \textbf{sSet}$, we then have that $R(J)$ is the functor
$$
    J_0 \leftarrow J_1 \times_{J_0} J_0^{\textbf{nerve}(I[1])} \leftarrow J_2 \times_{J_1 \times_{J_0} J_0^{\textbf{nerve}(I[1])}} (J_1 \times_{J_0} J_0^{\textbf{nerve}(I[1])})^{\textbf{nerve}(I[1])} \leftarrow \cdots
$$
Writing down $R(J)_n$ without recursively referring to $R^{\leq n - 1}(J)$ would appear to be difficult; naively `flattening out' the definition produces an exponentially growing expression of nested pullbacks. There is however a simpler non-inductive way to write down this functor, taking inspiration from the construction of homotopy limits in \cite[Def. 18.1.8]{hirschhornModelCategoriesTheir2009}; indeed, if we restrict $J$ to $M^{\leq n}$, then we have that $\lim_{M^{\leq n}} R(J) \cong R(J)_n$, so each $R(J)_n$ should be an inductively growing model of homotopy limit:

\begin{proposition}
    Suppose $J : M \rightarrow \textbf{sSet}$ is as above. Then for any $n \geq 0$, we have
    $$
        R(J)_n \cong \textbf{eq} \Big( \prod_{i = 0}^n J_i^{\textbf{nerve}(I[n-i])} \rightrightarrows \prod_{j=0}^{n-1} J_j^{\textbf{nerve}(I[n-j-1])} \Big)
    $$
    where the first map acts by precomposition with $\textbf{nerve}(I[n-j-1]) \rightarrow \textbf{nerve}(I[n-j])$ sending $i \mapsto i$ on objects, while the second map acts by postcomposition with $J_j \rightarrow J_{j-1}$ if $j > 0$ and the terminal map $J_0 \rightarrow \ast$.

    Moreover, the maps $R(J)_n \rightarrow R(J)_{n-1}$ are given by commutative diagrams
    \[\begin{tikzcd}
    	{\prod_{i = 0}^n J_i^{\textbf{nerve}(I[n-i])}} & {\prod_{j = 0}^{n-1} J_j^{\textbf{\textbf{nerve}}(I[n-j-1])}} \\
    	{\prod_{i = 0}^{n-1} J_i^{\textbf{nerve}(I[n-i-1])}} & {\prod_{j = 0}^{n-2} J_j^{\textbf{\textbf{nerve}}(I[n-j-2])}}
    	\arrow[shift left, from=1-1, to=1-2]
    	\arrow[shift right, from=1-1, to=1-2]
    	\arrow[from=1-1, to=2-1]
    	\arrow[from=1-2, to=2-2]
    	\arrow[shift left, from=2-1, to=2-2]
    	\arrow[shift right, from=2-1, to=2-2]
    \end{tikzcd}\]
    where the leftmost vertical map is given on $J_i^{\textbf{nerve}(I[n-i])}$ for $i < n$ by precomposition with maps $I[n-i-1] \rightarrow I[n-i]$ sending $j \mapsto j+1$ for objects $j$ and on $J_n$ by the terminal map $J_n \rightarrow \ast$. The rightmost vertical map is similar.

    Finally, the maps $(\kappa_J)(n) : J_n \rightarrow R(J)_n$ are given by the maps
    $$
        J_n \rightarrow J_i \rightarrow J_i^{\textbf{nerve}(I[n-i])}.
    $$
\end{proposition}

\begin{proof}
    The case $n = 0$ is trivial. Suppose then that the result holds for $R(J)_{n-1}$. We have that $R(J)_n$ is isomorphic to the limit of the diagram
    \begin{align*}
        &J_n \times_{\textbf{eq} \Big( \prod_{i = 0}^{n-1} J_i^{\textbf{nerve}(I[n-i-1])} \rightrightarrows \prod_{j = 0}^{n-2} J_j^{\textbf{\textbf{nerve}}(I[n-j-2])} \Big)} \\
        &\textbf{eq} \Big( \prod_{i = 0}^{n-1} J_i^{\textbf{nerve}(I[n-i-1])} \rightrightarrows \prod_{j = 0}^{n-2} J_j^{\textbf{\textbf{nerve}}(I[n-j-2])} \Big)^{\textbf{nerve}(I[1])} \\
        \cong &J_n \times_{\textbf{eq} \Big( \prod_{i = 0}^{n-1} J_i^{\textbf{nerve}(I[n-i-1])} \rightrightarrows \prod_{j = 0}^{n-2} J_j^{\textbf{\textbf{nerve}}(I[n-j-2])} \Big)} \\
        &\textbf{eq} \Big( \prod_{i = 0}^{n-1} J_i^{\textbf{nerve}(I[n-i-1] \times I[1])} \rightrightarrows \prod_{j = 0}^{n-2} J_j^{\textbf{\textbf{nerve}}(I[n-j-2] \times I[1])} \Big).
    \end{align*}
    Note that
    $$
        J_n \cong \textbf{eq} \Big( \prod_{i = 0}^n J_i \rightrightarrows \prod_{j=0}^{n-1} J_j \Big)
    $$
    where the two maps are induced by the diagrams
    \[\begin{tikzcd}
    	{J_n} & {J_{n-1}} & \cdots & {J_1} & {J_0} \\
    	\ast & {J_{n-1}} & \cdots & {J_1} & {J_0}
    	\arrow[from=1-1, to=2-1]
    	\arrow["id"', from=1-2, to=2-2]
    	\arrow["id"', from=1-4, to=2-4]
    	\arrow["id"', from=1-5, to=2-5]
    \end{tikzcd}\]
    and
    \[\begin{tikzcd}
    	{J_n} & {J_{n-1}} & \cdots & {J_1} & {J_0} \\
    	& {J_{n-1}} & \cdots & {J_1} & {J_0} & \ast
    	\arrow[from=1-1, to=2-2]
    	\arrow[from=1-2, to=2-3]
    	\arrow[from=1-3, to=2-4]
    	\arrow[from=1-4, to=2-5]
    	\arrow[from=1-5, to=2-6]
    \end{tikzcd}\]
    Thus, since limits commute, we may reinterpret $R(J)_n$ as the equalizer
    \begin{align*}
        \textbf{eq} \Big( & J_n \times \prod_{i = 0}^{n-1} J_i \times_{J_i^{\textbf{nerve}(I[n-i-1])}} J_i^{\textbf{nerve}(I[n-i-1] \times I[1])} \\
        \rightrightarrows & J_{n-1} \times \prod_{i=0}^{n-2} J_i \times_{J_i^{\textbf{nerve}(I[n-i-2])}} J_i^{\textbf{nerve}(I[n-i-2] \times I[1])} \Big)
    \end{align*}
    The induction now holds, by the isomorphisms
    $$
        J_i^{\textbf{nerve}(I[n-i])} \cong J_i \times_{J_i^{\textbf{nerve}(I[n-i-1])}} J_i^{\textbf{nerve}(I[n-i-1] \times I[1])}
    $$
    given by the maps $I[n-i] \rightarrow \ast$ and by $I[n-i] \rightarrow I[n-i-1] \times I[1]$ sending $0 \mapsto (0, 0)$ and $j \mapsto (j-1, 1)$ for $0 < j \leq n-i$. The maps $R(J)_n \rightarrow R(J)_{n-1}$ are as needed by inspection, as are the maps $(\kappa_J)_n : J_n \rightarrow R(J)_n$; indeed, the matching object $M_n R^{\leq n-1}(J)$ is just $R(J)_{n-1}$.
\end{proof}

We have that the $0$-simplices of $R(J)_n$ correspond naturally to tuples
$$
    (p_n \in J_n, p_{n-1} : \textbf{nerve}(I[1]) \rightarrow J_{n-1}, \cdots, p_0 : \textbf{nerve}(I[n]) \rightarrow J_0)
$$
of maps $p_i : \textbf{nerve}(I[n-i]) \rightarrow J_i$, where $p_i|_{\textbf{nerve}(I[n-i-1])}$, restricted under the map $\textbf{nerve}(I[n-i-1]) \rightarrow \textbf{nerve}(I[n-i])$ sending $j \mapsto j$ for all $j \in I[n-i-1]$, is equal to the composition
$$
    \textbf{nerve}(I[n-i-1]) \xrightarrow{p_{i+1}} J_{i+1} \rightarrow J_i
$$
for all $i < n$. That is to say, a $0$-simplex corresponds to a commutative diagram
\[\begin{tikzcd}[sep=small]
	{\ast } & {\textbf{nerve}(I[0])} & {\textbf{nerve}(I[1])} & \cdots & {\textbf{nerve}(I[n-1])} & {\textbf{nerve}(I[n])} \\
	& {J_n} & {J_{n-1}} & \cdots & {J_1} & {J_0}
	\arrow["\cong", from=1-1, to=1-2]
	\arrow[from=1-2, to=1-3]
	\arrow["{p_n}"', from=1-2, to=2-2]
	\arrow[from=1-3, to=1-4]
	\arrow["{p_{n-1}}"', from=1-3, to=2-3]
	\arrow[from=1-4, to=1-5]
	\arrow[from=1-5, to=1-6]
	\arrow["{p_1}"', from=1-5, to=2-5]
	\arrow["{p_0}", from=1-6, to=2-6]
	\arrow[from=2-2, to=2-3]
	\arrow[from=2-3, to=2-4]
	\arrow[from=2-4, to=2-5]
	\arrow[from=2-5, to=2-6]
\end{tikzcd}\]

Taking the limit $\lim_M R(J)$ in turn yields a Kan complex whose $0$-simplices are infinite sequences of finite tuples
$$
    ((p_0^0), (p_0^1, p_1^1), (p_0^2, p_1^2, p_2^2), (p_0^3, p_1^3, p_2^3, p_3^3), \cdots)
$$
such that $p_i^i \in (J_i)_0$ for all $i \geq 0$, while the maps
$$
    p_j^i : \textbf{nerve}(I[i-j]) \rightarrow J_j
$$
restrict along the maps $\psi_j^i : I[i-j-1] \rightarrow I[i-j]$ sending $k \mapsto k$ for $k \in I[i-j-1]$ to $p_{j+1}^i$ composed with the map $J_{j+1} \rightarrow J_j$. Moreover, considering the maps $\phi_j^i : I[i-j-1] \rightarrow I[i-j]$ sending $k \mapsto k+1$, we have that the diagrams
\[\begin{tikzcd}
	{\textbf{nerve}(I[i-j-1])} && {\textbf{nerve}(I[i-j])} \\
	& {J_j}
	\arrow["{\textbf{nerve}(\phi_j^i)}", from=1-1, to=1-3]
	\arrow["{p_j^{i-1}}"', from=1-1, to=2-2]
	\arrow["{p_j^i}", from=1-3, to=2-2]
\end{tikzcd}\]
commute.

What we have then is, for each $i \geq 0$, a commutative diagram of the form
\[\begin{tikzcd}
	{\textbf{nerve}(I[0])} && {\textbf{nerve}(I[1])} && {\textbf{nerve}(I[2])} & \cdots \\
	{J_i}
	\arrow["{\textbf{nerve}(\phi^{i+1}_i)}", from=1-1, to=1-3]
	\arrow["{p^i_i}"', from=1-1, to=2-1]
	\arrow["{\textbf{nerve}(\phi^{i+2}_i)}", from=1-3, to=1-5]
	\arrow["{p^{i+1}_i}"', from=1-3, to=2-1]
	\arrow[from=1-5, to=1-6]
	\arrow["{p^{i+2}_i}", from=1-5, to=2-1]
\end{tikzcd}\]
Let $\mathbb{Z}_{\geq 0}$ be equipped with the natural poset category structure. Considering the functor $\phi^i : \mathbb{Z}_{\geq 0} \rightarrow \textbf{Cat}$ sending $[n] \mapsto I[n]$ and the morphism $a \leq a+1$ for $a \in \mathbb{Z}_{\geq 0}$ to $\phi^{i+a}_i$, we could alternatively represent this as a map
$$
    p^i : \text{colim}_{\mathbb{Z}_{\geq 0}} \textbf{nerve}(\phi^i) \rightarrow J_i.
$$
The object $\text{colim}_{\mathbb{Z}_{\geq 0}} \textbf{nerve}(\phi^i)$ is naturally identified with the nerve $\textbf{nerve}(I[\infty])$ of the contractible groupoid $I[\infty]$ whose objects are the elements of $\mathbb{Z}_{\geq 0}$. The maps $I[n] \rightarrow I[\infty]$ then send $k \mapsto n-k$ for all $n \geq 0$.

Consider then the map $\psi : \textbf{nerve}(I[\infty]) \rightarrow \textbf{nerve}(I[\infty])$ induced by the maps $\psi^i_j$. Concretely, this map sends $n \mapsto n+1$ for all $n \in \mathbb{Z}_{\geq 0}$. By the first commutativity condition, we have that a $0$-simplex of $\lim_M R(J)$ precisely corresponds to a commutative diagram of the form
\[\begin{tikzcd}
	{\textbf{nerve}(I[\infty])} & {\textbf{nerve}(I[\infty])} & {\textbf{nerve}(I[\infty])} & \cdots \\
	{J_0} & {J_1} & {J_2} & \cdots
	\arrow["{p^0}"', from=1-1, to=2-1]
	\arrow["\psi"', from=1-2, to=1-1]
	\arrow["{p^1}"', from=1-2, to=2-2]
	\arrow["\psi"', from=1-3, to=1-2]
	\arrow["{p^2}"', from=1-3, to=2-3]
	\arrow[from=1-4, to=1-3]
	\arrow[from=2-2, to=2-1]
	\arrow[from=2-3, to=2-2]
	\arrow[from=2-4, to=2-3]
\end{tikzcd}\]
This seems analogous to another description of homotopy limits along the same diagram $M$, though with target category $\textbf{Top}$, in \cite[Ex. 6.5.6]{riehlCategoricalHomotopyTheory2014}; here, Riehl explains that the homotopy limit of a functor $M \rightarrow \textbf{Top}$ where the maps $J_{i+1} \rightarrow J_i$ are inclusions can be described by a map $f : [0, \infty) \rightarrow J_0$ such that $f([n, \infty)) \subseteq J_n$ for all $n \geq 0$.

The case of $\mathscr{C} = M$ seems in particular to suggest that, given a levelwise fibrant functor $X : \mathscr{C} \rightarrow \textbf{sSet}$, the replacement $R(X)(x)$ for each $x \in \mathscr{C}$ should be at least strongly related to homotopy limits of $X$ over each category $(x \downarrow \mathscr{C}^-)$. Indeed, note that for any $x \in \mathscr{C}$, the natural restriction $R(X)|_{(x \downarrow \mathscr{C}^-)}$ is in fact a Reedy fibrant replacement for $X|_{(x \downarrow \mathscr{C}^-)}$ and thus an injective fibrant replacement. Hence, $R(X)_x$ must be a homotopy limit of the diagram $X|_{(x \downarrow \mathscr{C}^-)}$. This suggests that what we have constructed levelwise for $R(X)$ is a series of inductively defined homotopy limits over finite-degree Reedy categories. It would therefore be of interest to move past our inductive definition of $R$ to see $R(X)$ explicitly at each level. We leave this task to future work.

\printbibliography

\end{document}